\documentclass{article}
\title{Large deviations for extreme eigenvalues of deformed Wigner random matrices}
\author{Benjamin M\textsuperscript{c}Kenna}
\usepackage{amsmath, amsthm, amsfonts, amssymb, mathrsfs, enumerate, color, tikz, fullpage}
\usepackage{stmaryrd}
\usepackage{bbm}
\usepackage{framed}

\usepackage[inline]{asymptote}

\usepackage{bm} 
\usepackage{tocloft} 
\usepackage{cite} 
\usepackage{subcaption} 

\usetikzlibrary{arrows,%
                petri,%
                topaths, automata}%

\usepackage{graphicx}
\numberwithin{equation}{section}

\makeatletter
\renewcommand{\section}{\@startsection
{section}
{1}
{0mm}
{-2\baselineskip}
{1\baselineskip}
{\normalfont\large\scshape\centering}} 
\makeatother

\makeatletter
\renewcommand{\subsection}{\@startsection
{subsection}
{2}
{0mm}
{-\baselineskip}
{0 \baselineskip}
{\normalfont\bf\itshape}} 
\makeatother

\def\@empty{}

\def\author#1{\par
    {\centering{\authorfont#1}\par\vspace*{0.05in}}
}

\def\titlefont{\fontsize{13}{15}\bfseries\boldmath\selectfont\centering{}}
\def\authorfont{\fontsize{13}{15}}
\def\abstractfont{\fontsize{8}{10}}

\let\affiliationfont\rhfont

\def\address#1{\par
    {\centering{\affiliationfont#1\par}}\par\vspace*{11pt}
}

\def\body{
\setcounter{footnote}{0}
\def\thefootnote{\alph{footnote}}
\def\@makefnmark{{$^{\rm \@thefnmark}$}}
}

\def\title#1{
    \thispagestyle{plain}
    \vspace*{-14pt}
    \vskip 79pt
    {\centering{\titlefont #1\par}}%
    \vskip 1em
}

\renewenvironment{abstract}{\par%
    \vspace*{6pt}\noindent 
    \abstractfont
    \noindent\leftskip10pt\rightskip10pt
}{%
  \par}

\newcommand{\C}{\mathbb{C}}

\newcommand{\E}{\mathbb{E}}
\newcommand{\N}{\mathbb{N}}
\renewcommand{\P}{\mathbb{P}}

\newcommand{\R}{\mathbb{R}}

\newcommand{\abs}[1]{\left\vert #1 \right\vert}

\newcommand{\mc}[1]{\mathcal{#1}}

\newcommand{\ms}[1]{\mathscr{#1}}
\newcommand{\ol}[1]{\overline{#1}}

\newcommand{\ip}[1]{\left\langle #1 \right\rangle}
\newcommand{\diff}{\mathop{}\!\mathrm{d}}

\DeclareMathOperator{\tr}{tr}
\DeclareMathOperator{\supp}{supp}

\DeclareMathOperator{\Id}{Id}

\DeclareMathOperator{\diag}{diag}

\newtheorem{thm}{Theorem}[section]
\newtheorem{cor}[thm]{Corollary}
\newtheorem{prop}[thm]{Proposition}
\newtheorem{lem}[thm]{Lemma}

\newtheorem{quest}[thm]{Question}

\theoremstyle{definition}
\newtheorem{defn}[thm]{Definition}

\newtheorem{assn}{Assumption}

\theoremstyle{remark}
\newtheorem{rem}[thm]{Remark}

\theoremstyle{definition}
\swapnumbers
\newtheorem{innercustomhyp}{Hypothesis}
\newenvironment{customhyp}[1]
  {\renewcommand\theinnercustomhyp{#1}\innercustomhyp}
  {\endinnercustomhyp}

\makeatletter
\let\c@equation\c@thm
\makeatother
\numberwithin{equation}{section}

\begin{document}

~\vspace{-1.4cm}

\title{Large deviations for extreme eigenvalues of deformed Wigner random matrices}

\vspace{1cm}
\noindent

\begin{minipage}[b]{0.3\textwidth}
\hspace{3cm}
 
 \end{minipage}
 \begin{minipage}[b]{0.3\textwidth}
 \author{Benjamin M\textsuperscript{c}Kenna}

\address{Courant Institute\\
   E-mail: mckenna@cims.nyu.edu}
 \end{minipage}

\begin{minipage}[b]{0.3\textwidth}

 \end{minipage}

\begin{abstract}
We present a large deviation principle at speed $N$ for the largest eigenvalue of some additively deformed Wigner matrices. In particular this includes Gaussian ensembles with full-rank general deformation. For the non-Gaussian ensembles, the deformation should be diagonal, and we assume that the laws of the entries have sharp sub-Gaussian Laplace transforms and satisfy certain concentration properties. For these latter ensembles we establish the large deviation principle in a restricted range $(-\infty,x_c)$, where $x_c$ depends on the deformation only and can be infinite. 
\end{abstract}

\vspace{-0.3cm}

\tableofcontents

\vspace{0.2cm}

\section{Introduction}\ 

\vspace{-0.5cm}

\subsection{Deformed ensembles: typical behavior.}\
In this paper, our goal is to prove a large deviation principle (LDP) for the largest eigenvalue of the random matrix
\begin{equation}
\label{eqn:model}
    X_N = \frac{W_N}{\sqrt{N}} + D_N.
\end{equation}
Here $\tfrac{W_N}{\sqrt{N}}$ lies in a particular class of real or complex Wigner matrices. Specifically, we will ask that the laws of the entries of $W_N$ have sub-Gaussian Laplace transforms with certain variances, and that these laws satisfy concentration properties. The archetypal examples of this class are the Gaussian ensembles (GOE and GUE). We also assume that $D_N$ is a deterministic matrix whose empirical spectral measure tends to a deterministic limit $\mu_D$ and whose extreme eigenvalues tend to the edges of $\mu_D$. In all of our proofs we will assume that $D_N$ is diagonal, but by rotational invariance, our results hold for the deformed Gaussian models even when $D_N$ is not diagonal. More details on our assumptions will be given in Section \ref{sec:setup}.

If we write $\lambda_1(M) \leq \cdots \leq \lambda_N(M)$ for the eigenvalues of a self-adjoint matrix $M$ and $\hat{\mu}_{M} = \frac{1}{N} \sum_{i=1}^N \delta_{\lambda_i(M)}$ for its empirical measure, it is well-known that 
\[
    \hat{\mu}_{X_N} \to \rho_{\text{sc}} \boxplus \mu_D,
\]
both almost surely and in expectation, where $\rho_{\text{sc}}$ is the semicircle law normalized as $\rho_{\text{sc}}(\diff x) = \tfrac{1}{2\pi}\sqrt{(4-x^2)_+} \diff x$ and $\mu \boxplus \nu$ is the free convolution of the probability measures $\mu$ and $\nu$ \cite{Pas1972, Voi1991}.

If $\mu$ is a compactly supported measure on $\R$, we write $\mathtt{l}(\mu)$ and $\mathtt{r}(\mu)$ for the left and right endpoints, respectively, of its support. For some special cases of our model, it is known that
\[
    \lambda_N(X_N) \to \mathtt{r}(\rho_{\text{sc}} \boxplus \mu_D) \quad \text{almost surely.}
\]
New cases will be a corollary of our large deviation principle; see Remark \ref{rem:lmax_converges} below for details.

Our model also exhibits edge universality for many choices of $D_N$; that is, the fluctuations of $\lambda_N(X_N)$, rescaled appropriately, are known to follow the Tracy-Widom distribution. This was first established by \cite{Shc2011} for the deformed GUE, if $\hat{\mu}_{D_N} \to \mu_D$ quickly ($d(\hat{\mu}_{D_N},\mu_D) = O(N^{-2/3-\epsilon})$ is enough, where $d$ is defined in Equation \eqref{eqn:distance}) and without outliers. The convergence-rate assumption was removed by \cite{CapPec2016}, which also allowed a finite number of outliers in a controlled way, under a technical assumption implying that $\mu_D$ does not decay too quickly near its edges. The assumption of Gaussianity was removed by \cite{LeeSch2015}, under a similar technical assumption on $\mu_D$.

\subsection{History of large deviations in random matrix theory.}\
The history of LDPs for random matrix theory is fairly sparse. The first result, from \cite{BenGui1997}, is for the empirical measure of the Gaussian ensembles. The first LDP for the largest eigenvalue of a random matrix ensemble, namely for the GOE, appeared in \cite{BenDemGui2001}. We mention also \cite{FeyVanKlo2008} for the largest eigenvalue of thin sample covariance matrices, and \cite{BorCap2014} for the empirical measure and \cite{Aug2016} for the largest eigenvalue of Wigner matrices whose entries have tails heavier than Gaussian.

There are also several results for the large deviations of deformed random matrices. For example, the paper \cite{GuiZei2002} studied large deviations of the empirical measure of full-rank deformations of Gaussian ensembles, making rigorous a prediction from \cite{Mat1994}. The largest eigenvalue of a rank-one deformation of a Gaussian ensemble was studied by \cite{Mai2007}; this result was recovered as the time-one marginal of a large deviation principle for Hermitian Brownian motions in \cite{DonMai2012}. Finite-rank deformations, rather than rank-one deformations, were covered in \cite{BenGuiMai2012}.

Our work builds on the recent papers \cite{GuiHui2018} and \cite{GuiMai2018}. These works use techniques discussed below to establish LDPs for extreme eigenvalues, treating respectively sharp sub-Gaussian Wigner matrices and the free-convolution model $A+UBU^\ast$ (with $U$ Haar orthogonal or Haar unitary). This method was also adapted in \cite{BirGui2019} to study joint large deviations of the largest eigenvalue and of one component of the corresponding eigenvector for rank-one deformations of Gaussian ensembles. Very recently, \cite{AugGuiHus2019} adapted this method to study non-sharp sub-Gaussian Wigner matrices; see Remark \ref{rem:ssg} below for a precise explanation of this terminology.

\subsection{Large deviations for ensembles with full-rank deformations.}\
In many large-deviations proofs, one wants to tilt measures by a Laplace transform. The insight of the paper \cite{GuiHui2018} was that the appropriate Laplace transform in our context is the so-called \emph{(rank-one) spherical integral}
\begin{equation}
\label{eqn:introduce_sph_int}
	\E_e[e^{N\theta \ip{e,Me}}].
\end{equation}
Here $M$ is an $N \times N$ self-adjoint matrix, $\theta \geq 0$ is the argument of the Laplace transform, and the integration $\E_e$ is over vectors $e$ uniform on the unit sphere $\mathbb{S}^{N-1}$ (we take $\mathbb{S}^{N-1} \subset \R^N$ if $M$ is real, or $\mathbb{S}^{N-1} \subset \C^N$ if $M$ is complex, so that \eqref{eqn:introduce_sph_int} is real). If $M$ is a random matrix, then \eqref{eqn:introduce_sph_int} is a random variable. This is a special case of the famous Harish-Chandra/Itzykson/Zuber integral.

For an LDP for the model \eqref{eqn:model}, we encounter two technical challenges. If we write $\P_N$ for the law of $X_N$ and $\E_{X_N}$ for the corresponding expectation (and define $\E_{W_N}$ in the obvious way), then the main challenge is the computation of
\begin{align}
\label{eqn:introduce_free_energy}
    \lim_{N \to \infty} \frac{1}{N} \log \E_{X_N}[\E_e[e^{N\theta\ip{e,X_Ne}}]] = \lim_{N \to \infty} \frac{1}{N} \log \E_e[\E_{W_N}[e^{\sqrt{N}\theta\ip{e,W_Ne}}] \cdot e^{N\theta\ip{e,D_Ne}}].
\end{align}
The term $\E_{X_N}[\E_e[e^{N\theta\ip{e,X_Ne}}]]$ appears as a normalization constant when tilting the measure, so its logarithmic asymptotics appear as part of the rate function. To understand these asymptotics when $W_N$ is not Gaussian, we use the method of \cite[Lemma 3.2]{GuiHui2018} to understand $\E_{W_N}[e^{\sqrt{N}\theta\ip{e,W_Ne}}]$ pointwise for unit vectors $e$ that are delocalized in an appropriate sense. We combine this with the new result (see Lemma \ref{lem:delocalized_sph_int} below) 
\[
    \text{for $\theta$ small enough depending on $\mu_D$,} \qquad \lim_{N \to \infty} \frac{1}{N} \log \left[ \frac{\E_e[\mathbf{1}_{e \text{ delocalized}} e^{N\theta\ip{e,D_Ne}}]}{\E_e[e^{N\theta\ip{e,D_Ne}}]}\right] = 0.
\]
The qualifier ``for $\theta$ small enough'' means that, via this argument, we can only obtain large-deviations asymptotics of events that localize $\lambda_N(X_N)$ below some critical threshold $x_c$, which depends on the deformation $\mu_D$ only. We show $x_c \geq \mathtt{r}(\rho_{\text{sc}} \boxplus \mu_D)$ with strict inequality except in degenerate cases, and that $x_c$ can be infinite. For example, $x_c = +\infty$ when $\mu_D$ is the uniform measure on an interval. For the Gaussian ensembles, the limit in \eqref{eqn:introduce_free_energy} is directly computable for every $\theta \geq 0$ without recourse to this delocalization problem, so our results for those models are stronger.

The second difficulty is that we need a concentration result of the form
\begin{equation}
\label{eqn:introduce_concentration}
    \lim_{N \to \infty} \frac{1}{N} \log \P_N(d(\hat{\mu}_{X_N}, \rho_{\text{sc}} \boxplus \mu_D) > N^{-\kappa}) = -\infty
\end{equation}
for $\kappa > 0$ small enough, where $d$ is defined in \eqref{eqn:distance}. With $\rho_{\text{sc}} \boxplus \mu_D$ replaced with $\E[\hat{\mu}_{X_N}]$, this is standard concentration of linear statistics \cite{GuiZei2000}, easily extended to our model. To approximate $\E[\hat{\mu}_{X_N}]$ with $\rho_{\text{sc}} \boxplus \mu_D$, we use local laws for deformed ensembles \cite{LeeSchSteYau2016, LeeSch2015, ErdKruSch2019}. Our argument is slightly technical, since these local laws let us approximate $\E[\hat{\mu}_{X_N}]$, not directly by $\rho_{\text{sc}} \boxplus \mu_D$, but by a measure close to $\rho_{\text{sc}} \boxplus \hat{\mu}_{D_N}$, so several intermediate comparisons are needed.

The organization of the paper is as follows: In Section \ref{sec:setup}, we state our assumptions and main result with commentary and examples. In Section \ref{sec:overview}, we provide background on spherical integrals, introduce the tilted measures, and provide a high-level overview of the technique as well as proofs of weak-large-deviations upper and lower bounds. These arguments rely on several key lemmas, the proofs of which make up the remaining three sections. In Section \ref{sec:free_energy}, we address the first technical issue discussed above. In Section \ref{sec:concentration}, we prove exponential tightness for our model, then address the second technical issue discussed above. In Section \ref{sec:rate_fn}, we establish properties of the rate function. Throughout, our results are stated for both the real and complex cases, but we only give proofs in the real case. The proofs in the complex case require only minor modifications.

\subsection*{Conventions.}\
We use the shorthand $\beta$ for the symmetry class at hand: $\beta = 1$ refers to real symmetric matrices and $\beta = 2$ refers to complex Hermitian matrices. Our norm $\|M\|$ on matrices is the operator norm $\|M\| = \sup_{\|u\|_2 = 1} \|Mu\|_2$. Our metric $d$ on probability measures will be the Dudley distance (also called the bounded-Lipschitz distance), given by
\begin{equation}
\label{eqn:distance}
    d(\mu,\nu) = \sup\left\{ \abs{\int f d(\mu - \nu)} : \sup_{x \neq y} \frac{\abs{f(x)-f(y)}}{\abs{x-y}} + \|f\|_{L^\infty} \leq 1\right\}.
\end{equation}
Recall that this distance metrizes weak convergence.

Finally, we recall the Stieltjes transform and the Voiculescu $R$-transform of a compactly supported probability measure. If $\mu$ is a probability measure on $\R$ the convex hull of whose support is $[a,b]$, then we will normalize its Stieltjes transform $G_\mu$ as
\[
	G_\mu(y) = \int \frac{\mu(\diff t)}{y-t}.
\]
If we write $G_\mu(a) = \lim_{y \uparrow a} G_\mu(y)$ and $G_\mu(b) = \lim_{y \downarrow b} G_\mu(y)$, then it can be shown that $G_\mu$ is a bijection from $\R \setminus [a,b]$ to $(G_\mu(a),G_\mu(b)) \setminus \{0\}$. We will write 
\[
    K_\mu : (G_\mu(a),G_\mu(b)) \setminus \{0\} \to \R \setminus [a,b]
\]
for its functional inverse, and write
\[
	R_\mu(y) = K_\mu(y) - \frac{1}{y}
\]
for its Voiculescu $R$-transform, which linearizes free convolution: $R_{\mu \boxplus \nu} = R_{\mu} + R_{\nu}$.

\subsection*{Acknowledgements.}\
The author would like to thank Paul Bourgade for many helpful discussions, and Alice Guionnet and Ofer Zeitouni for explaining that one assumption in an early version of this paper was superfluous.


\section{Main result}
\label{sec:setup}

\subsection{Assumptions.}\
We first present our assumptions on $D_N$, which will be made throughout, even though we will only state them in the presentation of the main results.

\begin{assn}
\label{assns:basic}
The matrix $D_N$ is real, diagonal, and deterministic, and its empirical measure $\hat{\mu}_{D_N}$ tends weakly as $N \to \infty$ to a compactly supported probability measure $\mu_D$. Furthermore,
\begin{align*}
    \lambda_N(D_N) \to \mathtt{r}(\mu_D), \\
    \lambda_1(D_N) \to \mathtt{l}(\mu_D).
\end{align*}
\end{assn}

\begin{assn}
\label{assns:quantitative}
There exist $C > 0$ and $\epsilon_0 > 0$ such that
\[
    d(\hat{\mu}_{D_N}, \mu_{D}) \leq CN^{-\epsilon_0}.
\]
\end{assn}

\begin{rem}
We emphasize that $\mu_D$ is allowed to be quite poorly behaved. For example, it can be singular with respect to Lebesgue measure. It can also have disconnected support. Notice that Assumption \ref{assns:quantitative} is fairly mild. For example, if $\mu_D$ has a density and the entries of $D_N$ are the $\frac{1}{N}$-quantiles of $\mu_D$, then in fact $d(\hat{\mu}_{D_N},\mu_D) = O(\tfrac{1}{N})$. If the entries of $D_N$ were obtained from i.i.d. random variables, we would have $d(\hat{\mu}_{D_N},\mu_D) = O(\tfrac{1}{\sqrt{N}})$. 

In fact, the proof of Lemma \ref{lem:cdf_diff} below shows that, instead of Assumption \ref{assns:quantitative}, it suffices to bound the difference between the Stieltjes transforms of $\hat{\mu}_{D_N}$ and $\mu_D$ at distance $N^{-\delta}$ from the real line, for $\delta > 0$ small enough.
\end{rem}

We will write the Laplace transform of a measure $\mu$ on $\C$ as
\[
    T_{\mu}(t) := \int e^{\Re(z\ol{t})} \mu(\diff z).
\]
If in fact $\mu$ is supported on $\R$ and $t$ is real, this reduces to the familiar
\[
    T_{\mu}(t) = \int e^{tx} \mu(\diff x).
\]

We assume that $\tfrac{W_N}{\sqrt{N}}$ is a Wigner matrix, by which we mean that its entries are independent up to the self-adjoint condition. Our assumptions on the Wigner part are named, rather than numbered, to emphasize that our results apply under either of them, rather than both of them. 

\begin{customhyp}{Gaussian}
\label{hyp:Gaussian}
The matrix $\frac{W_N}{\sqrt{N}}$ is distributed according to the Gaussian Orthogonal Ensemble if $\beta = 1$, or the Gaussian Unitary Ensemble if $\beta = 2$. (That is, the law of $W_N$ on the space of symmetric/Hermitian matrices has density proportional to $\exp(-\beta\tr(W_N^2)/4)$.)
\end{customhyp}

\begin{customhyp}{SSGC}
\label{hyp:SSGC}
(This labelling stands for ``sharp sub-Gaussian and concentrates,'' and matches the assumptions of \cite{GuiHui2018}.)

Write $\mu_{i,j}^N$ for the law of the $(i,j)$th entry of $W_N$.
\begin{enumerate}
\item Assume \textbf{both} of the following. 
\begin{itemize}
\item The first and second moments match those of the relevant Gaussian ensemble. In our normalization, this means that for every $N \in \N$ and $i, j \in \llbracket 1, N \rrbracket$, if $\beta = 1$ we have
\[
    \int x \mu_{i,j}^N (\diff x) = 0, \quad \int x^2 \mu_{i,j}^N (\diff x) = 1+\delta_{ij},
\]
whereas if $\beta = 2$ and $i \neq j$ we have
\begin{align*}
    \int \Re(z) \mu_{i,j}^N(\diff z) &= \int \Im(z) \mu_{i,j}^N(\diff z) = \int \Re(z)\Im(z) \mu_{i,j}^N(\diff z) = 0, \\
    \int \Re(z)^2 \mu_{i,j}^N(\diff z) &= \int \Im(z)^2 \mu_{i,j}^N(\diff z) = \frac{1}{2}.
\end{align*}
If $\beta = 2$, then each $\mu_{i,i}^N$ is supported on $\R$, with $\int x \mu_{i,i}^N(\diff x) = 0$ and $\int x^2\mu_{i,i}^N(\diff x) = 1$. 
\item For every $N \in \N$ and $i,j \in \llbracket 1, N \rrbracket$, the measure $\mu_{i,j}^N$ has a \emph{sharp sub-Gaussian Laplace transform}:
\begin{equation}
\label{eqn:ssg}
    \text{for all } \begin{cases} \text{$t \in \R$ if $\beta = 1$} \\ \text{$t \in \C$ if $\beta = 2$} \end{cases}, \quad T_{\mu_{i,j}^N}(t) \leq \exp\left(\frac{\abs{t}^2(1+\delta_{ij})}{2\beta}\right).
\end{equation}
\end{itemize}
\item In addition, assume \textbf{one} of the following concentration-type hypotheses.
\begin{itemize}
\item There exists a constant $c$ independent of $N$ such that, for all $N \in \N$ and all $i,j \in \llbracket 1, N \rrbracket$, the law $\mu_{i,j}^N$ satisfies a log-Sobolev inequality with constant $c$.
\item There exists a compact set $K$ independent of $N$ (real if $\beta = 1$, or complex if $\beta = 2$) such that, for all $N \in \N$ and all $i,j \in \llbracket 1, N \rrbracket$, the law $\mu_{i,j}^N$ is supported in $K$.
\end{itemize}
\end{enumerate}
\end{customhyp}

\begin{rem}
\label{rem:ssg}
A list of examples satisfying the \ref{hyp:SSGC} Hypothesis is provided in \cite{GuiHui2018}. Among these examples are real matrices whose entries follow the Rademacher law $\frac{1}{2}(\delta_{-1}+\delta_{+1})$ or the uniform law on $[-\sqrt{3},\sqrt{3}]$ (appropriately rescaled on the diagonal).

In the literature, it is common to call a centered measure $\mu$ on $\R$ with unit variance \emph{sub-Gaussian} whenever
\[
    A := 2\sup_{t \in \R} \frac{1}{t^2} \log T_\mu(t)
\]
is finite. We emphasize that we are asking for more; in \eqref{eqn:ssg} we require $A = 1$ (off the diagonal, with appropriate modifications otherwise), and following \cite{GuiHui2018} we call such measures \emph{sharp sub-Gaussian}. This is a strict subclass; for example, the law of $\tfrac{1}{p}BG$, where $B \sim \text{Bernoulli}(p)$ and $G \sim \mc{N}(0,1)$ are independent, has unit variance but $A = 1/p$. This example appears in \cite{AugGuiHus2019}, which treats the general case $A > 1$, with zero deformation.
\end{rem}

\subsection{Main result.}\

\begin{defn}
For a compactly supported measure $\nu$, a parameter $\theta \geq 0$, and a real number $\mathscr{M} \geq \mathtt{r}(\nu)$, define
\begin{align}
\label{eqn:limitJ}
	J(\nu,\theta,\mathscr{M}) = J^{(\beta)}(\nu,\theta,\mathscr{M}) = 
	\begin{cases} 
	    \frac{\beta}{2} \int_0^{\frac{2}{\beta}\theta} R_\nu(t) \diff t & \text{if } 0 \leq \frac{2}{\beta}\theta \leq G_\nu(\mathscr{M}), \\
	    \theta\mathscr{M} - \frac{\beta}{2}\left[1+\log\left(\frac{2}{\beta}\theta\right)\right] - \frac{\beta}{2}\int \log(\mathscr{M}-y) \nu(\diff y) & \text{if } \frac{2}{\beta}\theta \geq G_\nu(\mathscr{M}).
	\end{cases}
\end{align}
(If $\ms{M} = \mathtt{r}(\nu)$, we recall our convention $G_\nu(\mathtt{r}(\nu)) = \lim_{y \downarrow \mathtt{r}(\nu)} G_\nu(y)$, which is possibly infinite.) In Section \ref{sub:sph_int} we will explain how this function arises as the limit of appropriately normalized spherical integrals.

For $x \geq \mathtt{r}(\rho_{\text{sc}} \boxplus \mu_D)$ and $\theta \geq 0$, we define
\[
    I^{(\beta)}(x,\theta) = J(\rho_{\text{sc}} \boxplus \mu_D, \theta, x) - \frac{\theta^2}{\beta} - J(\mu_D,\theta,\mathtt{r}(\mu_D))
\]
and then set
\[
    I^{(\beta)}(x) = \begin{cases} 
        +\infty & \text{if } x < \mathtt{r}(\rho_{\text{sc}} \boxplus \mu_D), \\
        \sup_{\theta \geq 0} I^{(\beta)}(x,\theta) & \text{if } x \geq \mathtt{r}(\rho_{\text{sc}} \boxplus \mu_D). \end{cases}
\]
\end{defn}
We will show below that
\[
    I^{(2)}(x) = 2I^{(1)}(x)
\]
for all measures $\mu_D$.

To state our result, we will need the following critical threshold.

\begin{defn}
Given the compactly supported measure $\mu_D$, define the real number $x_c$ by
\begin{align*}
    x_c &= x_c(\mu_D) = \begin{cases} \mathtt{r}(\mu_D) + G_{\mu_D}(\mathtt{r}(\mu_D)) & \text{if } G_{\mu_D}(\mathtt{r}(\mu_D)) < +\infty, \\ +\infty & \text{otherwise}. \end{cases}
\end{align*}
It will be shown in Proposition \ref{prop:study_of_rate_fn} below that $x_c \geq \mathtt{r}(\rho_{\text{sc}} \boxplus \mu_D)$, with equality if and only if an inequality involving the Stieltjes transform of $\mu_D$ degenerates.
\end{defn}

The main result of the paper is the following:

\begin{thm}
\label{thm:ldp}
Suppose that Assumptions \ref{assns:basic} and \ref{assns:quantitative} hold.
\begin{enumerate}
\item If the \ref{hyp:Gaussian} Hypothesis holds, then the law of the largest eigenvalue $\lambda_N(X_N)$ satisfies a large deviation principle at speed $N$ with the good rate function $I^{(\beta)}(x)$. By rotational invariance, we have the same result when $D_N$ is not diagonal but simply symmetric (if $\beta = 1$) or Hermitian (if $\beta = 2$) and satisfies the rest of the requirements of Assumption \ref{assns:basic}.
\item If instead the \ref{hyp:SSGC} Hypothesis holds, then the law of the largest eigenvalue $\lambda_N(X_N)$ satisfies what we will call a ``restricted large deviation principle on $(-\infty,x_c)$'' at speed $N$ with the good rate function $I^{(\beta)}(x)$. This means the following:
\begin{itemize}
\item For every closed set $F \subset (-\infty,x_c)$, we have
\begin{equation}
\label{eqn:restricted_ldp_ub}
    \limsup_{N \to \infty} \frac{1}{N} \log \P_N(\lambda_N(X_N) \in F) \leq -\inf_{x \in F} I^{(\beta)}(x).
\end{equation}
\item For every open set $G \subset (-\infty,x_c)$, we have 
\begin{equation}
\label{eqn:restricted_ldp_lb}
    \liminf_{N \to \infty} \frac{1}{N} \log \P_N(\lambda_N(X_N) \in G) \geq -\inf_{x \in G} I^{(\beta)}(x).
\end{equation}
\end{itemize}
\item In particular, if the \ref{hyp:SSGC} Hypothesis holds and $\mu_D$ is such that $x_c = +\infty$, then the law of the largest eigenvalue $\lambda_N(X_N)$ satisfies a large deviation principle at speed $N$ with the good rate function $I^{(\beta)}(x)$ in the usual sense.
\end{enumerate}
\end{thm}

\begin{rem}
\label{rem:lmax_converges}
See Proposition \ref{prop:study_of_rate_fn} below for a more in-depth study of the function $I^{(\beta)}(x)$. There, it is shown that $I^{(\beta)}(x)$ has a unique minimizer at $x = \mathtt{r}(\rho_{\text{sc}} \boxplus \mu_D)$, where it takes the value zero. In particular, if the \ref{hyp:Gaussian} Hypothesis holds, or if the \ref{hyp:SSGC} Hypothesis holds and $\mu_D$ is such that $x_c = +\infty$, then 
\begin{equation}
\label{eqn:aslimit}
    \lambda_N(X_N) \to \mathtt{r}(\rho_{\text{sc}} \boxplus \mu_D) \quad \text{almost surely.}
\end{equation}
This result appears to be new in the real case when $\rho_{\text{sc}} \boxplus \mu_D$ is multicut, and in the complex non-Gaussian case when $\rho_{\text{sc}} \boxplus \mu_D$ is multicut and $(D_N)_{N=1}^\infty$ has ``internal outliers'' between the connected components of $\supp(\mu_D)$ that persist as $N \to \infty$. (Recall that we forbid ``external outliers'' by assuming $\lambda_N(D_N) \to \mathtt{r}(\mu_D)$ and $\lambda_1(D_N) \to \mathtt{l}(\mu_D)$.) In the literature Equation \eqref{eqn:aslimit} appears as an easy corollary of edge universality results, or as a special case of BBP results when the deforming matrix $D_N$ has no external outliers. For example, it follows from \cite{CapPec2016} for deformed GUE, possibly multicut with internal outliers, under some assumptions about the decay rate of $\mu_D$ near its edges; from \cite{LeeSch2015} for general real or complex noise if $\mu_D$ is such that $\rho_{\text{sc}} \boxplus \mu_D$ is supported on a single interval with square-root decay at its two edges; and from \cite{BelCap2017} in the complex (and possibly multicut) case with no outliers. Of course, all of these papers achieve much more.
\end{rem}

\begin{rem}
The proof of the ``restricted LDP,'' i.e., of Equations \eqref{eqn:restricted_ldp_lb} and \eqref{eqn:restricted_ldp_ub}, follows in the classical way from estimates of small-ball probabilities via a weak large deviation principle and exponential tightness, except that we can only lower-bound small-ball probabilities $\P_N(\abs{\lambda_N(X_N) - x} < \delta)$ for $x < x_c$ rather than $x \in \R$. However, we can upper-bound these probabilities for all $x$ (see Theorem \ref{thm:wkldpub}), so Equation \eqref{eqn:restricted_ldp_ub} actually holds for all closed $F \subset \R$, not just $F \subset (-\infty,x_c)$. 
\end{rem}

\begin{rem}
Of course, one would prefer to write the rate function non-variationally, and we can do this when the argument is at or above the critical threshold $x_c(\mu_D)$. Proposition \ref{prop:study_of_rate_fn} shows that, for all $x > \mathtt{r}(\rho_{\text{sc}} \boxplus \mu_D)$, the supremum in the definition of $I(x)$ is achieved at a unique $\theta_x$. For $x \geq x_c$ (which is relevant for the Gaussian case), this $\theta_x$ is given explicitly as $\theta_x = \tfrac{\beta}{2}(x-\mathtt{r}(\mu_D))$; thus if $x \geq x_c(\mu_D)$,
\begin{align*}
    \text{if } x \geq x_c(\mu_D), \qquad I^{(\beta)}(x) = \frac{\beta}{2}\left[\frac{(x-\mathtt{r}(\mu_D))^2}{2} - \int\log(x-y)(\rho_{\text{sc}} \boxplus \mu_D)(\diff y) + \int \log(\mathtt{r}(\mu_D)-y)\mu_D(\diff y) \right].
\end{align*}
(If $x_c(\mu_D) < \infty$, then $\int \log(\mathtt{r}(\mu_D) - y)\mu_D(\diff y) < \infty$.) But for subcritical $x$ values, $\theta_x$ is defined implicitly in the proof of Proposition \ref{prop:study_of_rate_fn} as the unique solution of the constrained problem
\begin{equation}
\label{eqn:constrained_thetax}
    \frac{2}{\beta} \theta_x + K_{\mu_D}\left(\frac{2}{\beta} \theta_x\right) = x \qquad \text{subject to} \qquad \theta_x \in \left(\frac{\beta}{2}G_{\rho_{\text{sc}} \boxplus \mu_D}(\mathtt{r}(\rho_{\text{sc}} \boxplus \mu_D)),\frac{\beta}{2}G_{\mu_D}(\mathtt{r}(\mu_D))\right).
\end{equation}
We have not found a way to solve this constrained problem explicitly, nor to write $I^{(\beta)}(x,\theta_x)$ explicitly at its solution. If the domain of $\theta_x$ in the constraint were instead $(0,\frac{\beta}{2}G_{\rho_{\text{sc}} \boxplus \mu_D}(\mathtt{r}(\rho_{\text{sc}} \boxplus \mu_D)))$, the equation would simplify to $K_{\rho_{\text{sc}} \boxplus \mu_D}(\tfrac{2}{\beta}\theta_x) = x$, which has the solution $\theta_x = \frac{\beta}{2}G_{\rho_{\text{sc}} \boxplus \mu_D}(x)$. But $K_{\rho_{\text{sc}} \boxplus \mu_D}(\cdot)$ is not generally guaranteed to exist for arguments larger than $G_{\rho_{\text{sc}} \boxplus \mu_D}(\mathtt{r}(\rho_{\text{sc}} \boxplus \mu_D))$, and even when extendable it may not be globally invertible.

Thus our rate function remains implicit for subcritical $x$ values. Nevertheless, in some simple cases the constrained problem can be solved explicitly; two examples are given below in Sections \ref{sub:first_example} and \ref{sub:second_example}.
\end{rem}

\begin{rem}
If $D_N = 0$, then $x_c = +\infty$,
\begin{align*}
    I^{(\beta)}(x) &= \begin{cases} +\infty & \text{if } x < 2 \\ \sup_{\theta \geq 0}\left\{J(\rho_{\text{sc}},\theta,x) - \frac{\theta^2}{\beta}\right\} & \text{if } x \geq 2 \end{cases} = \begin{cases} +\infty & \text{if } x < 2 \\ \frac{\beta}{2} \int_2^x \sqrt{t^2-4} \diff t & \text{if } x \geq 2, \end{cases}
\end{align*}
and we recover \cite[Theorems 1.4 and 1.5]{GuiHui2018}, which in particular includes the classical LDP for the Gaussian ensembles. (The last equality in the above display is true by \cite[Section 4.1]{GuiHui2018}.) Notice that we get the same rate function if $D_N$ is not identically zero but rather $\|D_N\| \to 0$ sufficiently quickly.
\end{rem}

\begin{rem}
One wants to recover large deviations for BBP-type problems, so it is tempting to conjecture that, if the largest eigenvalue of $D_N$ tends not to $\mathtt{r}(\mu_D)$ but to some $\rho > \mathtt{r}(\mu_D)$, then an LDP should hold for $\lambda_N(X_N)$ at speed $N$ with the good rate function
\[
    \tilde{I}^{(\beta)}(x) = \begin{cases} +\infty & \text{if } x < \mathtt{r}(\rho_{\text{sc}} \boxplus \mu_D) \\ \sup_{\theta \geq 0} \left\{J(\rho_{\text{sc}} \boxplus \mu_D, \theta, x) - \frac{\theta^2}{\beta} - J(\mu_D,\theta,\rho)\right\} & \text{otherwise}. \end{cases}
\]
But, at least for certain simple situations, such a conjecture would be wrong. For example, suppose that $\tfrac{W_N}{\sqrt{N}}$ is distributed according to the GOE (if $\beta = 1$) or the GUE (if $\beta = 2$), that $\mu_D = \delta_0$ (so that $\rho_{\text{sc}} \boxplus \mu_D = \rho_{\text{sc}}$), and that $D_N$ has $N-1$ zero eigenvalues with one spike at, say, $2$ for concreteness. Then it is known \cite[Theorem 1.2]{Mai2007} that $\lambda_N(X_N)$ satisfies an LDP at speed $N$ with the good rate function 
\[
    \hat{I}^{(\beta)}(x) = \begin{cases} +\infty & x < 2 \\ \frac{\beta}{4}\int_{\frac{5}{2}}^x \sqrt{z^2-4} \diff z - \beta\left(x-\frac{5}{2}\right)+\frac{\beta}{8}\left[x^2-\left(\frac{5}{2}\right)^2\right] & x \geq 2. \end{cases}
\]
(The published rate function has a typo; it is corrected in the v2 arXiv posting. We also normalize the semicircle law differently.) Notice that this vanishes uniquely at $x = \tfrac{5}{2}$, which lies outside $\supp(\rho_{\text{sc}})$ -- this model is past the BBP phase transition. But in this situation we can compute
\[
    \tilde{I}^{(\beta)}(x) = \begin{cases} +\infty & x < 2 \\ 0 & 2 \leq x \leq \frac{5}{2} \\ \hat{I}^{(\beta)}(x) & x \geq \frac{5}{2} \end{cases}
\]
It is likely that our method could be extended, as in \cite{GuiMai2018}, to models where $\lim_{N \to \infty} \lambda_N(D_N)$ is a spike below the BBP threshold, i.e., such that still $\lambda_N(X_N) \to \mathtt{r}(\rho_{\text{sc}} \boxplus \mu_D)$ almost surely. But a new idea is needed beyond the BBP threshold.
\end{rem}

\subsection{First example ($x_c < \infty$).}\
\label{sub:first_example}
If 
\[
    \mu_D(\diff x) = \frac{1}{2\pi \sigma^2} \sqrt{4\sigma^2-x^2} \mathbf{1}_{x \in [-2\sigma,2\sigma]} \diff x
\]
for some parameter $\sigma > 0$, then $\rho_{\text{sc}} \boxplus \mu_D$ is again semicircular, scaled so its support lies in $[-2\sqrt{\sigma^2+1},2\sqrt{\sigma^2+1}]$. The constrained equation $\eqref{eqn:constrained_thetax}$ can be solved explicitly, and we can calculate
\[
    I^{(\beta)}(x) = \begin{cases} +\infty & \text{if } x < \mathtt{r}(\rho_{\text{sc}} \boxplus \mu_D) = 2\sqrt{\sigma^2+1} \\ \beta\left[\frac{x\sqrt{x^2-4(1+\sigma^2)}}{4(1+\sigma^2)} + \log \left( \frac{2\sqrt{1+\sigma^2}}{x+\sqrt{x^2-4(1+\sigma^2)}} \right)\right] & \text{if } 2\sqrt{\sigma^2+1} \leq x \leq 2\sigma+\frac{1}{\sigma} = x_c \\ 
    \beta\left[\frac{(x-2\sigma)^2}{4} + \frac{x\sqrt{x^2-4(1+\sigma^2)}-x^2}{8(1+\sigma^2)} + \frac{1}{2} \log \left( \frac{2\sigma}{x+\sqrt{x^2-4(1+\sigma^2)}} \right) + \frac{1}{2}\right] & \text{if } x \geq 2\sigma+\frac{1}{\sigma} = x_c. \end{cases}
\]
Notice that $I^{(\beta)}(x)$ is $C^2$ but no better at $x_c$, which is perhaps surprising. Figure \ref{fig:semicircle} plots this function when $\beta = 1$ and $\sigma = 1$ (i.e., when $\mu_D$ is the usual semicircle law supported on $[-2,2]$). Here $\mathtt{r}(\rho_{\text{sc}} \boxplus \mu_D) = 2\sqrt{2} \approx 2.83$, $x_c = 3$, and $I^{(\beta)}(x_c) \approx 0.03 \cdot \beta$. Under the \ref{hyp:SSGC} Hypothesis, we would be able to estimate, say, $\P_N(\lambda_N \in (2.9,2.95))$ but not $\P_N(\lambda_N \in (2.9,3.1))$.

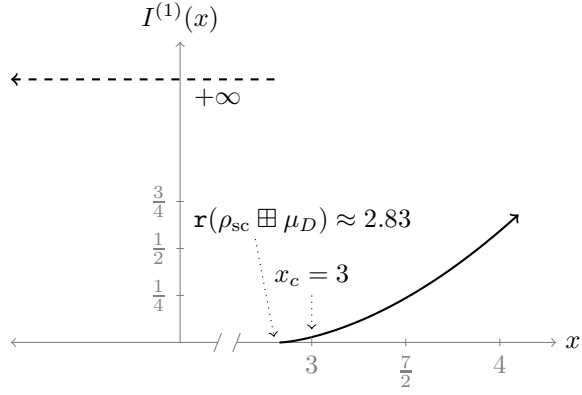
\begin{figure}[h!]
\begin{center}
\begin{tikzpicture}[scale=2.5]
    \draw [<-, gray] (2.3,1.6) -- (2.3,0);
    \draw [<-, gray] (1.4,0) -- (2.5,0);
    \draw [gray] (2.48,-0.05) -- (2.52,0.05);
    \draw [gray] (2.58,-0.05) -- (2.62,0.05);
    \draw [->, gray] (2.6,0) -- (4.3,0);
    \node [above] at (2.3,1.6) {$I^{(1)}(x)$};
    \node [right] at (4.3,0) {$x$};
    \draw [gray] (2.28,0.25) -- (2.32,0.25);
    \draw [gray] (2.28,0.5) -- (2.32,0.5);
    \draw [gray] (2.28,0.75) -- (2.32,0.75);
    \draw [gray] (3.5,0.02) -- (3.5,-0.02);
    \draw [gray] (4,0.02) -- (4,-0.02);
    \node [gray, left] at (2.28,0.25) {$\frac{1}{4}$};
    \node [gray, left] at (2.28,0.5) {$\frac{1}{2}$};
    \node [gray, left] at (2.28,0.75) {$\frac{3}{4}$};
    \node [gray, below] at (3,-0.02) {$3$};
    \node [gray, below] at (3.5,-0.02) {$\frac{7}{2}$};
    \node [gray, below] at (4,-0.02) {$4$};
    \draw [<-, dashed, thick] (1.4,1.4) -- (2.82843,1.4);
    \node [below] at (2.5,1.4) {$+\infty$};
    \draw [black, thick, domain = 2.82843:3] plot(\x, {\x*sqrt(\x*\x-8)/8+ln(sqrt(8)/(\x+sqrt(\x*\x-8)))});
    \draw [black, thick, domain = 3:4.1, ->] plot(\x, {(\x-2)*(\x-2)/4+(\x*sqrt(\x*\x-8)-\x*\x)/16+ln(2/(\x+sqrt(\x*\x-8)))/2+1/2});
    \draw [->, dotted] (2.7,0.55) -- (2.8,0.03);
    \node at (2.95,0.65) {$\mathtt{r}(\rho_{\text{sc}} \boxplus \mu_D) \approx 2.83$};
    \draw [gray] (3,0.02) -- (3,-0.02);
    \draw [->, dotted] (3,0.25) -- (3,0.07);
    \node [above] at (3,0.25) {$x_c = 3$};
\end{tikzpicture}
\end{center}
\caption{Sketch of the rate function when $\beta = 1$ and $\mu_D = \rho_{\text{sc}}$.}
\label{fig:semicircle}
\end{figure}

\subsection{Second example ($x_c = \infty$).}\
\label{sub:second_example}
Now suppose
\[
    \mu_D = \frac{1}{2}(\delta_{-a}+\delta_{+a})
\]
for some parameter $a > 0$. Here $x_c(\mu_D) = G_{\mu_D}(a) = +\infty$, so all $x$ are subcritical; that is, we can estimate any probability $\P_N(\lambda_N \in A)$ under either the \ref{hyp:SSGC} Hypothesis or the \ref{hyp:Gaussian} Hypothesis. Our computations use the known result
\begin{equation}
\label{eqn:example_right_edge}
    \mathtt{r}(\rho_{\text{sc}} \boxplus \mu_D) = \frac{(4a^2-1+\sqrt{8a^2+1})^{3/2}}{2\sqrt{2}a(\sqrt{8a^2+1}-1)} =: \mathtt{r}(a).
\end{equation}
In the physics literature this dates back to \cite[Equations (55), (56)]{Zee1996}; it was established in the mathematical literature in \cite[Equations (3.5), (3.6)]{BleKui2004} (for $a > 1$), \cite[Section 1]{AptBleKui2005} (for $a < 1$), and \cite[Section 7]{BleKui2007} (for $a = 1$). The latter three papers establish that the measure $\rho_{\text{sc}} \boxplus \mu_D$ undergoes a phase transition at $a = 1$. When $a > 1$, the support of $\rho_{\text{sc}} \boxplus \mu_D$ consists of two intervals; when $a = 1$, these intervals meet at zero, where the density has cubic-root decay; and when $a < 1$ the support is a single interval, on the interior of which the density is strictly positive. (This set of three papers also establishes universality of correlation functions.) We emphasize that our results apply to all $a > 0$.

Using the equivalent \cite[Theorem 6]{GuiMai2005} formula
\[
    J(\nu,\theta,\ms{M}) = \theta R_\nu\left(\frac{2}{\beta}\theta\right) -\frac{\beta}{2} \int \log\left(1+\frac{2}{\beta}\theta R_\nu\left(\frac{2}{\beta}\theta\right)-\frac{2}{\beta}\theta y\right) \nu(\diff y),
\]
valid if $0 \leq \frac{2}{\beta} \theta \leq G_\nu(\ms{M})$, and the constrained equation \eqref{eqn:constrained_thetax} implicitly defining $\theta_x$, one can see that
\begin{align}
\label{eqn:rate_fn_example}
\begin{split}
    I^{(\beta)}(x) &= \frac{\theta_x^2}{\beta} + \frac{\beta}{2} \int \log\left(x-\frac{2}{\beta}\theta_x - y\right)\mu_D(\diff y) - \frac{\beta}{2} \int \log(x-y) (\rho_{\text{sc}} \boxplus \mu_D)(\diff y) \\
    &= \frac{\beta}{4}\left(\frac{2}{\beta}\theta_x\right)^2 + \frac{\beta}{4}  \log\left[ \left(x-\frac{2}{\beta}\theta_x\right)^2-a^2 \right] - \frac{\beta}{2} \int \log(x-y) (\rho_{\text{sc}} \boxplus \mu_D)(\diff y)
\end{split}
\end{align}
for $x \geq \mathtt{r}(\rho_{\text{sc}} \boxplus \mu_D)$. We invert $K_{\rho_{\text{sc}} \boxplus \mu_D}(y) = \frac{\sqrt{1+4a^2y^2}+2y^2+1}{2y}$ to obtain $G_{\rho_{\text{sc}} \boxplus \mu_D}(y)$ for $y > \mathtt{r}(a)$, choosing branches according to the requirement that $G_{\rho_{\text{sc}} \boxplus \mu_D}(y)$ be decreasing on $(\mathtt{r}(\rho_{\text{sc}} \boxplus \mu_D),\infty)$; this yields
\begin{align*}
    G_{\rho_{\text{sc}} \boxplus \mu_D}(y) = \frac{2}{3}\left[y- \sqrt{-3+3a^2+y^2}\sin\left(\frac{\pi}{3}-\frac{1}{3}\arctan\left(\frac{9y+18a^2y-2y^3}{\sqrt{-4(3-3a^2-y^2)^3-(9y+18a^2y-2y^3)^2}}\right)\right)\right] \\
\end{align*}
if $y > \mathtt{r}(a)$. In the limit $y \downarrow \mathtt{r}(a)$ we obtain
\[
    G_{\rho_{\text{sc}} \boxplus \mu_D}(\mathtt{r}(a)) = \frac{(-1+4a^2+\sqrt{1+8a^2})^{3/2}-\sqrt{(1+8a^2)(-1-4a^2+8a^4+\sqrt{1+8a^2})}}{3\sqrt{2}a(-1+\sqrt{1+8a^2})} =: \mathtt{c}(a).
\]
This gives us the bounds on the constrained problem \eqref{eqn:constrained_thetax}; since $K_{\mu_D}(y) = \frac{\sqrt{1+4a^2y^2}+1}{2y}$, this has the solution
\[
    \frac{2}{\beta}\theta_x = \frac{2}{3}\left[x - \sqrt{-3+3a^2+x^2}\sin\left(\frac{1}{3}\arctan\left(\frac{9x+18a^2x-2x^3}{\sqrt{-4(3-3a^2-x^2)^3-(9x+18a^2x-2x^3)^2}}\right)\right)\right] \quad \text{if } x > \mathtt{r}(a).
\]

On the other hand, since $\rho_{\text{sc}} \boxplus \mu_D$ decays at most like a cube root near its edges \cite[Corollary 5]{Bia1997}, we can differentiate under the integral sign to obtain
\[
    \int \log(x-y)(\rho_{\text{sc}} \boxplus \mu_D)(\diff y) = \int_{\mathtt{r}(\rho_{\text{sc}} \boxplus \mu_D)}^x G_{\rho_{\text{sc}} \boxplus \mu_D}(t) \diff t + \int \log(\mathtt{r}(\rho_{\text{sc}} \boxplus \mu_D)-y) (\rho_{\text{sc}} \boxplus \mu_D)(\diff y).
\]
We compute the second term on the right-hand side by setting $x = \mathtt{r}(\rho_{\text{sc}} \boxplus \mu_D)$ in \eqref{eqn:rate_fn_example}, since then $I^{(\beta)}(x) = 0$ and $\frac{2}{\beta}\theta_x = G_{\rho_{\text{sc}} \boxplus \mu_D}(\mathtt{r}(\rho_{\text{sc}} \boxplus \mu_D)) = \mathtt{c}(a)$; this yields
\[
    \int \log(\mathtt{r}(\rho_{\text{sc}} \boxplus \mu_D)-y) (\rho_{\text{sc}} \boxplus \mu_D)(\diff y) = \frac{1}{2}(\mathtt{c}(a)^2+\log((\mathtt{r}(a)-\mathtt{c}(a))^2-a^2)).
\]
Thus, if $x > \mathtt{r}(\rho_{\text{sc}} \boxplus \mu_D)$, 
\begin{align*}
    I^{(\beta)}(x)& \\
    = \frac{\beta}{4}\Bigg[&\left[ \frac{2}{3}\left[x - \sqrt{-3+3a^2+x^2}\sin\left(\frac{1}{3}\arctan\left(\frac{9x+18a^2x-2x^3}{\sqrt{-4(3-3a^2-x^2)^3-(9x+18a^2x-2x^3)^2}}\right)\right)\right]\right]^2  \\
    &+ \log\left[ \left( \frac{x}{3} + \left[ \frac{2\sqrt{-3+3a^2+x^2}}{3}\sin\left(\frac{1}{3}\arctan\left(\frac{9x+18a^2x-2x^3}{\sqrt{-4(3-3a^2-x^2)^3-(9x+18a^2x-2x^3)^2}}\right)\right)\right]\right)^2 - a^2 \right] \\
    &- 2 \int_{\mathtt{r}(a)}^x \frac{2}{3}\left[t-\sqrt{-3+3a^2+t^2}\sin\left(\frac{\pi}{3}-\frac{1}{3}\arctan\left(\frac{9t+18a^2t-2t^3}{\sqrt{-4(3-3a^2-t^2)^3-(9t+18a^2t-2t^3)^2}}\right)\right)\right] \diff t \\
    &- (\mathtt{c}(a)^2+\log((\mathtt{r}(a)-\mathtt{c}(a))^2-a^2))\Bigg].
\end{align*}
Figure \ref{fig:rademacher} plots this function at the critical parameter $a = 1$ (so that $\mathtt{r}(\rho_{\text{sc}} \boxplus \mu_D) = \frac{3\sqrt{3}}{2}$) when $\beta = 1$.

\begin{figure}[h!]
\begin{center}
\begin{tikzpicture}[scale=2.5]
    \draw [<-, gray] (2.1,1.6) -- (2.1,0);
    \draw [<-, gray] (1.4,0) -- (2.3,0);
    \draw [gray] (2.28,-0.05) -- (2.32,0.05);
    \draw [gray] (2.38,-0.05) -- (2.42,0.05);
    \draw [->, gray] (2.4,0) -- (4,0);
    \node [above] at (2.1,1.6) {$I^{(1)}(x)$};
    \node [right] at (4,0) {$x$};
    \draw [<-, dashed, thick] (1.4,1.4) -- (2.59808,1.4);
    \node [below] at (2.4,1.4) {$+\infty$};
    \draw [->,dotted] (2.7,-0.4) -- (2.6,-0.04);
    \node at (2.7,-0.5) {$\mathtt{r}(\rho_{\text{sc}} \boxplus \mu_D) \approx 2.60$};
    \draw [gray] (2.5,0.02) -- (2.5,-0.02);
    \draw [gray] (3,0.02) -- (3,-0.02);
    \draw [gray] (3.5,0.02) -- (3.5,-0.02);
    \draw [gray] (2.08,0.25) -- (2.12,0.25);
    \draw [gray] (2.08,0.5) -- (2.12,0.5);
    \draw [gray] (2.08,0.75) -- (2.12,0.75);
    \node [gray, left] at (2.1,0.25) {$\frac{1}{4}$};
    \node [gray, left] at (2.1,0.5) {$\frac{1}{2}$};
    \node [gray, left] at (2.1,0.75) {$\frac{3}{4}$};
    \node [gray, below] at (2.5,-0.02) {$\frac{5}{2}$};
    \node [gray, below] at (3,-0.02) {$3$};
    \node [gray, below] at (3.5,-0.02) {$\frac{7}{2}$};
    \draw [thick, ->] (2.59808, 0) -- (2.61, 0.000762498) -- (2.62, 0.00190294) -- (2.63, 0.00334702) -- (2.64, 0.00504208) -- 
    (2.65, 0.00695665) -- (2.66, 0.00906917) -- (2.67, 0.0113637) -- (2.68, 0.0138279) -- (2.69, 0.0164517) --
    (2.70, 0.0192268) -- (2.71, 0.0221465) -- (2.72, 0.0252045) -- (2.73, 0.0283958) -- (2.74, 0.0317158) --
    (2.75, 0.0351604) -- (2.76, 0.0387259) -- (2.77, 0.0424092) -- (2.78, 0.0462072) -- (2.79, 0.0501172) --
    (2.80, 0.0541368) -- (2.81, 0.0582636) -- (2.82, 0.0624956) -- (2.83, 0.0668307) -- (2.84, 0.0712673) --
    (2.85, 0.0758034) -- (2.86, 0.0804377) -- (2.87, 0.0851685) -- (2.88, 0.0899945) -- (2.89, 0.0949144) --
    (2.90, 0.0999269) -- (2.91, 0.105031) -- (2.92, 0.110225) -- (2.93, 0.115509) -- (2.94, 0.12088) --
    (2.95, 0.126339) -- (2.96, 0.131885) -- (2.97, 0.137516) -- (2.98, 0.143231) -- (2.99, 0.149031) --
    (3.00, 0.154914) -- (3.01, 0.160879) -- (3.02, 0.166925) -- (3.03, 0.173053) -- (3.04, 0.179261) --
    (3.05, 0.185549) -- (3.06, 0.191916) -- (3.07, 0.198362) -- (3.08, 0.204886) -- (3.09, 0.211487) --
    (3.10, 0.218165) -- (3.11, 0.224919) -- (3.12, 0.23175) -- (3.13, 0.238656) -- (3.14, 0.245637) --
    (3.15, 0.252693) -- (3.16, 0.259823) -- (3.17, 0.267027) -- (3.18, 0.274304) -- (3.19, 0.281655) --
    (3.20, 0.289078) -- (3.21, 0.296573) -- (3.22, 0.30414) -- (3.23, 0.311779) -- (3.24, 0.319489) --
    (3.25, 0.32727) -- (3.26, 0.335121) -- (3.27, 0.343043) -- (3.28, 0.351035) -- (3.29, 0.359097) --
    (3.30, 0.367228) -- (3.31, 0.375428) -- (3.32, 0.383697) -- (3.33, 0.392034) -- (3.34, 0.40044) --
    (3.35, 0.408914) -- (3.36, 0.417456) -- (3.37, 0.426066) -- (3.38, 0.434743) -- (3.39, 0.443487) --
    (3.40, 0.452298) -- (3.41, 0.461176) -- (3.42, 0.47012) -- (3.43, 0.479131) -- (3.44, 0.488207) --
    (3.45, 0.49735) -- (3.46, 0.506558) -- (3.47, 0.515832) -- (3.48, 0.525171) -- (3.49, 0.534575) --
    (3.50, 0.544044) -- (3.51, 0.553578) -- (3.52, 0.563177) -- (3.53, 0.57284) -- (3.54, 0.582567) --
    (3.55, 0.592359) -- (3.56, 0.602214) -- (3.57, 0.612133) -- (3.58, 0.622116) -- (3.59, 0.632162) --
    (3.60, 0.642272) -- (3.61, 0.652445) -- (3.62, 0.662681) -- (3.63, 0.672979) -- (3.64, 0.683341) --
    (3.65, 0.693765) -- (3.66, 0.704252) -- (3.67, 0.714801) -- (3.68, 0.725412) -- (3.69, 0.736086) --
    (3.70, 0.746821);
\end{tikzpicture}
\end{center}
\caption{Sketch of the rate function when $\beta = 1$ and $\mu_D = \frac{1}{2}(\delta_{1}+\delta_{-1})$.}
\label{fig:rademacher}
\end{figure}
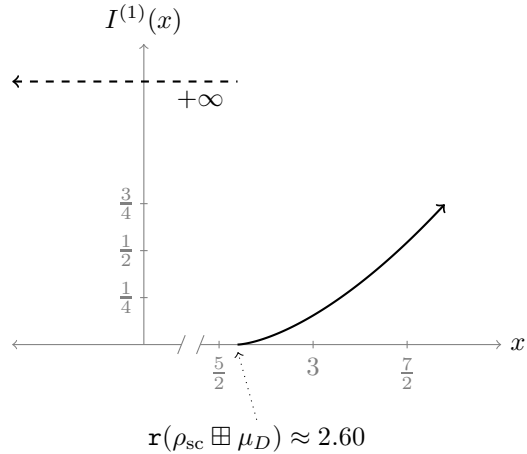

\begin{quest}
Does the mechanism driving the deviations $\{\lambda_N(X_N) \approx x\}$ change as $x$ passes the critical threshold $x_c$? Specifically, can one formalize and prove the notion that, with large probability, while the eigenvector corresponding to $\lambda_N$ is delocalized under the above event for subcritical $x$ values, it localizes for supercritical $x$ values?
\end{quest}


\section{Proof overview}
\label{sec:overview}

\subsection{Spherical integrals.}\
\label{sub:sph_int}
Given a self-adjoint $N \times N$ matrix $X$ and $\theta \geq 0$, consider
\begin{align*}
	I_N(X,\theta) &= \E_e[e^{N\theta\ip{e,Xe}}], \\
	J_N(X,\theta) &= \frac{1}{N} \log I_N(X,\theta).
\end{align*}
We recall that $\E_e$ only averages over the unit sphere, so if $X$ is random then $I_N(X,\theta)$ and $J_N(X,\theta)$ are random variables.

If $\{X_N\}$ is such that $\hat{\mu}_{X_N}$ has a weak limit $\nu$, then we might hope that $J_N(X_N,\theta)$ also has a limit depending on $\nu$ and $\theta$. This is so; but the limit also depends on $\lambda_N(X_N)$ if $\theta$ is sufficiently large. This should not be surprising, since the integrand $e^{N\theta\ip{e,Xe}}$ is maximized near the eigenvector corresponding to $\lambda_N(X)$, especially for larger $\theta$ values. Indeed, we have the following result.

\begin{prop}
\cite[Theorem 6]{GuiMai2005}
\label{prop:spherical_limit}
Suppose that the sequence $(A_N)_{N=1}^\infty$ of self-adjoint matrices is such that $\hat{\mu}_{A_N} \to \nu$ weakly for some compactly-supported measure $\nu$, that $\lambda_1(A_N)$ has a finite limit, and that $\lambda_N(A_N) \to \mathscr{M}$ for some real number $\mathscr{M}$. (Notice that we are not assuming that $\mathscr{M}$ is the right edge of $\nu$, but of course we must have $\mathscr{M} \geq \mathtt{r}(\nu)$.) If $\theta \geq 0$, then
\[
	\lim_{N \to \infty} J_N(A_N,\theta) = J(\nu,\theta,\mathscr{M}),
\]
where $J(\nu,\theta,\mathscr{M})$ is as in \eqref{eqn:limitJ}.
\end{prop}

\subsection{Tilted measures and weak large deviations.}\
Our general strategy will be to show a weak large deviation principle, as well as exponential tightness. In the proof of the weak-large-deviations lower bound for our measure of interest, we will actually need a weak-large-deviations \emph{upper bound} for the following family of measures. 

\begin{defn}
Given $\theta \geq 0$, we consider the ``tilted'' measure $\P_N^{\theta}$ on $N \times N$ matrices (symmetric if $\beta = 1$, or Hermitian if $\beta = 2$) whose density with respect to the law $\P_N$ of $X_N$ is given by
\[
    \frac{\diff \P_N^{\theta}}{\diff \P_N}(X) = \frac{I_N(X,\theta)}{\E_{X_N}(I_N(X_N,\theta))}.
\]
Notice from the definition of $I_N$ that $\P_N^0 = \P_N$. 
\end{defn}

We will need the following asymptotics of the free energy for this measure, with proof in Section \ref{sec:free_energy}.

\begin{prop}
\label{prop:free_energy}
Given the compactly supported measure $\mu_D$, define the threshold
\[
    \theta_c = \theta_c(\beta, \mu_D) = \begin{cases} \frac{\beta}{2}G_{\mu_D}(\mathtt{r}(\mu_D)) & \text{if } G_{\mu_D}(\mathtt{r}(\mu_D)) < +\infty, \\ +\infty & \text{otherwise.} \end{cases}
\]
Under the \ref{hyp:Gaussian} Hypothesis, choose any $\theta \geq 0$; or, under the \ref{hyp:SSGC} Hypothesis, choose any $0 \leq \theta < \theta_c$. Then
\[
    \lim_{N \to \infty} \frac{1}{N} \log \E_{X_N}[I_N(X_N,\theta)] = \frac{\theta^2}{\beta} + J(\mu_D,\theta,\mathtt{r}(\mu_D)).
\]
\end{prop}

We split up the weak-large-deviations upper and lower bounds as follows:

\begin{thm}
\label{thm:wkldpub}
First, let $x < \mathtt{r}(\rho_{\text{sc}} \boxplus \mu_D)$. Under either Hypothesis, choose any $\theta \geq 0$. Then
\[
    \lim_{\delta \to 0} \limsup_{N \to \infty} \frac{1}{N} \log \P_N^{\theta}(\abs{\lambda_N(X_N) - x} \leq \delta) = -\infty.
\]
Second, let $x \geq \mathtt{r}(\rho_{\text{sc}} \boxplus \mu_D)$. Under the \ref{hyp:Gaussian} Hypothesis, choose any $\theta \geq 0$; or, under the \ref{hyp:SSGC} Hypothesis, choose any $0 \leq \theta < \theta_c$. Then 
\[
    \limsup_{\delta \to 0} \limsup_{N \to \infty} \frac{1}{N} \log \P_N^\theta\left(\abs{\lambda_N(X_N) - x} \leq \delta\right) \leq -(I^{(\beta)}(x)-I^{(\beta)}(x,\theta)).
\]
Notice that $I^{(\beta)}(x,0) = 0$ for all measures $\mu_D$ and all $x \geq \mathtt{r}(\rho_{\text{sc}} \boxplus \mu_D)$. Thus when $\theta = 0$ we recover the weak large deviation upper bound for the measure of primary interest, under either Hypothesis.
\end{thm}

\begin{thm}
\label{thm:wkldplb}
Under the \ref{hyp:Gaussian} Hypothesis, choose any $x \in \R$; or, under the \ref{hyp:SSGC} Hypothesis, choose any $x < x_c$. Then 
\[
    \liminf_{\delta \to 0} \liminf_{N \to \infty} \frac{1}{N} \log \P_N(\abs{\lambda_N(X_N) - x} < \delta) \geq -I^{(\beta)}(x).
\]
\end{thm}

\subsection{Outline.}\
When estimating $\tfrac{1}{N} \log \P_N(\abs{\lambda_N(X_N) - x} \leq \delta)$ by tilting by spherical integrals, one wants to understand $J_N(X_N,\theta)$ on the event $\{\abs{\lambda_N(X_N) - x} \leq \delta\}$. To localize $J_N(X_N,\theta)$, one needs to control $\hat{\mu}_{X_N}$. Therefore one wants to find a set 
\[
    \mc{A}_{x,\delta}^M \subset \{\abs{\lambda_N(X_N) - x} \leq \delta\}
\]
of matrices with controlled empirical measures (which will turn out to depend on some $M \gg 1$) satisfying both of the following:
\begin{itemize}
\item On the one hand, $\mc{A}_{x,\delta}^M$ is a continuity set for spherical integrals, in the sense that we have a good enough understanding of $J_N(M,\theta)$ for $M \in \mc{A}_{x,\delta}^M$ to be able to estimate 
\[
    \frac{1}{N} \log \P_N(\mc{A}_{x,\delta}^M) \approx e^{-NI^{(\beta)}(x)}.
\]
\item On the other hand, $\mc{A}_{x,\delta}^M$ is not too much smaller than $\{\abs{\lambda_N(X_N) - x} \leq \delta\}$, in the sense that
\[
    \frac{1}{N} \log \P_N(\abs{\lambda_N(X_N) - x} \leq \delta) \approx \frac{1}{N} \log \P_N(\mc{A}_{x,\delta}^M).
\]
\end{itemize}

The next subsection first details the continuity result of \cite{Mai2007}, which helps us choose $\mc{A}_{x,\delta}^M$ while satisfying the first point, then states a proposition which we need to show that our choice satisfies the second point.

\subsection{Continuity of spherical integrals.}\

\begin{prop}
\cite[Proposition 2.1]{Mai2007}
\label{prop:sph_int_cts}
For any $\theta > 0$ and any $\kappa > 0$, there exists a function $g_{\kappa,\theta} : \R^+ \to \R^+$ going to zero at zero such that, for any $\delta > 0$ and $N$ large enough, if $B_N$ and $B'_N$ are sequences of matrices such that $d(\hat{\mu}_{B_N},\hat{\mu}_{B'_N}) < N^{-\kappa}$, $\abs{\lambda_N(B_N)-\lambda_N(B'_N)} < \delta$, $\sup_N\|B_N\| < \infty$, and $\sup_N \|B'_N\| < \infty$, then we have
\[
    \abs{J_N(B_N,\theta) - J_N(B'_N,\theta)} < g_{\kappa,\theta}(\delta).
\]
\end{prop}

This suggests that we introduce the following deterministic sets of $N \times N$ symmetric matrices. Fix once and for all a $\kappa$ satisfying Proposition  \ref{prop:negligible_all_theta}, below, and write $g_\theta$ for $g_{\frac{\kappa}{2},\theta}$; then for any $x \in \R$, $\delta > 0$, and $M > 0$, let
\[
    \mc{A}_{x,\delta}^M = \{X : \abs{\lambda_N(X) - x} < \delta, d(\hat{\mu}_X,\rho_{\text{sc}} \boxplus \mu_D) < N^{-\kappa}, \text{ and }  \|X\| \leq M\}.
\]

In the next few results, we discretize the measure $\rho_{\text{sc}} \boxplus \mu_D$ so that we can apply Proposition \ref{prop:sph_int_cts} and control $J_N(X_N,\theta)$ uniformly for $X_N \in \mc{A}_{x,\delta}^M$. 

\begin{lem}
\label{lem:discretize}
Fix $x \geq \mathtt{r}(\rho_{\text{sc}} \boxplus \mu_D)$ and $M \geq \max(x,\abs{\mathtt{l}(\rho_{\text{sc}} \boxplus \mu_D)})$. Then there exists a sequence of deterministic matrices $B'_N$ with the following properties:
\begin{itemize}
\item $\lambda_N(B'_N) = x$, 
\item $\sup_{N \geq 1} \|B'_N\| \leq M$, and
\item $d(\hat{\mu}_{B'_N},\rho_{\text{sc}} \boxplus \mu_D) \leq N^{-\kappa}$ for $N$ sufficiently large.
\end{itemize}
\end{lem}
\begin{proof}
Given $N$, define the $\frac{1}{N}$ quantiles $\{\gamma_j\}_{j=1}^N = \{\gamma_j^{(N)}\}_{j=1}^N$ of the measure $\rho_{\text{sc}} \boxplus \mu_D$ implicitly by
\[
    \frac{j}{N} = (\rho_{\text{sc}} \boxplus \mu_D)((-\infty,\gamma_j)).
\]
(This is possible since $\rho_{\text{sc}} \boxplus \mu_D$ admits a density \cite[Corollary 2]{Bia1997}.) Then let $B'_N = \text{diag}(\gamma_1, \ldots, \gamma_{N-1}, x)$. Since our distance on probability measures is defined with respect to bounded-Lipschitz test functions, it is easy to show that, in fact, $d(\hat{\mu}_{B'_N},\rho_{\text{sc}} \boxplus \mu_D) = O(\tfrac{1}{N})$.
\end{proof}

\begin{cor}
\label{cor:matrix_cty_of_spherical}
For every $\theta \geq 0$, $x \geq \mathtt{r}(\rho_{\text{sc}} \boxplus \mu_D)$, $\delta > 0$, and $M > \max(x+\delta,\abs{\mathtt{l}(\rho_{\text{sc}} \boxplus \mu_D)})$, we have
\[
    \limsup_{N \to \infty} \sup_{B_N \in \mc{A}_{x,\delta}^M} \abs{J_N(B_N,\theta) - J(\rho_{\text{sc}} \boxplus \mu_D,\theta,x)} \leq g_\theta(\delta).
\]
\end{cor}
\begin{proof}
Let $\{B'_N\}_{N=1}^\infty$ be as in Lemma \ref{lem:discretize}. Then whenever $B_N \in \mc{A}_{x,\delta}^M$ we have $d(\hat{\mu}_{B_N},\hat{\mu}_{B'_N}) \leq 2N^{-\kappa} \leq N^{-\frac{\kappa}{2}}$, and $\abs{\lambda_N(B_N) - \lambda_N(B'_N)} \leq \delta$, so that by Proposition \ref{prop:sph_int_cts} and by our definition of $g_\theta$
\[
    \sup_{B_N \in \mc{A}_{x,\delta}^M} \abs{J_N(B_N,\theta) - J_N(B'_N,\theta)} \leq g_\theta(\delta)
\]
for $N$ sufficiently large. In addition, by Proposition \ref{prop:spherical_limit} we have
\[
    \lim_{N \to \infty} \abs{J_N(B'_N,\theta) - J(\rho_{\text{sc}} \boxplus \mu_D,\theta,x)} = 0.
\]
The result follows.
\end{proof}

On the other hand, the result below shows that the restrictions we added to $\{X : \abs{\lambda_N(X) - x} < \delta\}$ to arrive at $\mc{A}_{x,\delta}^M$ have probability negligibly close to $1$ at the exponential scale. Notice that the first point is exponential tightness. The proof will make up Section \ref{sec:concentration}.

\begin{prop}
\label{prop:negligible_all_theta}
Assume either the \ref{hyp:Gaussian} Hypothesis or the \ref{hyp:SSGC} Hypothesis.
\begin{enumerate}
\item For every $\theta \geq 0$ we have
\begin{align*}
    \lim_{K \to \infty} \limsup_{N \to \infty} \frac{1}{N} \log \P_N^\theta(\lambda_N > K) &= -\infty, \\
    \lim_{K \to \infty} \limsup_{N \to \infty} \frac{1}{N} \log \P_N^\theta(\lambda_1 < -K) &= -\infty.
\end{align*}
\item For every $\theta \geq 0$ we have
\[
    \lim_{M \to \infty} \limsup_{N \to \infty} \frac{1}{N}\log \P_N^\theta (\|X_N\| > M) = -\infty.
\]
\item There exists $\gamma > 0$ such that, for any $0 < \kappa < \gamma$ and any $\theta \geq 0$,
\[
    \lim_{N \to \infty} \frac{1}{N} \log \P_N^\theta(d(\hat{\mu}_{X_N},\rho_{\text{sc}} \boxplus \mu_D) > N^{-\kappa}) = -\infty.
\]
\end{enumerate}
\end{prop}

Theorem \ref{thm:ldp} follows in the classical way from the exponential tightness above, the weak LDP upper bound (Theorem \ref{thm:wkldpub}), and the weak LDP lower bound (Theorem \ref{thm:wkldplb}). We now prove the latter two.

\subsection{The proof of the weak LDP upper bound.}\

\begin{lem}
\label{lem:tilted_ub}
Fix $y \geq \mathtt{r}(\rho_{\text{sc}} \boxplus \mu_D)$ and $M > y$ sufficiently large. Under the \ref{hyp:Gaussian} Hypothesis, choose any $\theta \geq 0$; or, under the \ref{hyp:SSGC} Hypothesis, choose any $0 \leq \theta < \theta_c$. Then
\[
    \limsup_{\delta \to 0} \limsup_{N \to \infty} \frac{1}{N} \log \P_N^\theta(\mc{A}_{y,\delta}^M) \leq -(I^{(\beta)}(y) - I^{(\beta)}(y,\theta)).
\]
\end{lem}
\begin{proof}
For any $\theta' \geq 0$, we have
\begin{align*}
    \P_N^\theta(\mc{A}_{y,\delta}^M) &= \frac{1}{\E_{X_N}[I_N(X_N,\theta)]} \E_{X_N}\left[ \mathbf{1}_{X_N \in \mc{A}_{y,\delta}^M} I_N(X_N,\theta)\frac{I_N(X_N,\theta')}{I_N(X_N,\theta')} \right] \\
    &\leq \frac{\E_{X_N}[I_N(X_N,\theta')]}{\E_{X_N}[I_N(X_N,\theta)]} \left( \sup_{X \in \mc{A}_{y,\delta}^M} I_N(X,\theta) \right) \left( \sup_{X \in \mc{A}_{y,\delta}^M} \frac{1}{I_N(X,\theta')} \right).
\end{align*}
Fix $\epsilon > 0$. By Corollary \ref{cor:matrix_cty_of_spherical} and Lemmas \ref{lem:free_energy_ub} (applied to $\theta'$, which is any nonnegative number) and \ref{lem:free_energy_lb} (applied to $\theta$, which is subcritical if necessary), if $M > y+\delta$ (true for small enough $\delta$ since $M > y$) and for $N$ sufficiently large depending on $\theta$, $\theta'$, and $\epsilon$, we thus have
\[
    \frac{1}{N}\log \P_N^\theta(\mc{A}_{y,\delta}^M) \leq I^{(\beta)}(y,\theta)-I^{(\beta)}(y,\theta') + 2g_\theta(\delta) + 2g_{\theta'}(\delta) + \epsilon.
\]
By taking $N \to \infty$, then $\delta \downarrow 0$, then $\epsilon \downarrow 0$, we obtain
\[
    \limsup_{\delta \downarrow 0} \limsup_{N \to \infty} \frac{1}{N} \log \P_N^\theta(\mc{A}_{y,\delta}^M) \leq -(I^{(\beta)}(y,\theta')-I^{(\beta)}(y,\theta))
\]
which gives us the result by optimizing over $\theta'$.
\end{proof}

\begin{proof}[Proof of Theorem \ref{thm:wkldpub}]
We first focus on the case when $x < \mathtt{r}(\rho_{\text{sc}} \boxplus \mu_D)$. For such an $x$, if $\delta$ is so small that $x+\delta < \mathtt{r}(\rho_{\text{sc}} \boxplus \mu_D)-\delta$, then whenever $\abs{\lambda_N(X_N)-x} \leq \delta$, the empirical spectral measure $\hat{\mu}_{X_N}$ does not charge $(\mathtt{r}(\rho_{\text{sc}} \boxplus \mu_D)-\delta,\mathtt{r}(\rho_{\text{sc}} \boxplus \mu_D))$. Hence $d(\hat{\mu}_{X_N},\rho_{\text{sc}} \boxplus \mu_D) \geq f(\delta)$ for some positive function $f$. Thus for such $\delta$ and for $N$ large enough we have
\[
    \frac{1}{N} \log \P_N^\theta \left( \abs{ \lambda_N (X_N) - x} \leq \delta\right) \leq \frac{1}{N} \log \P_N^\theta \left(d(\hat{\mu}_{X_N},\rho_{\text{sc}} \boxplus \mu_D) > N^{-\kappa}\right)
\]
which suffices in light of Proposition \ref{prop:negligible_all_theta}. Thus in the following it remains only to consider $x \geq \mathtt{r}(\rho_{\text{sc}} \boxplus \mu_D)$.

Fix $\theta \geq 0$, $\delta > 0$, $x > \mathtt{r}(\rho_{\text{sc}} \boxplus \mu_D)$, and a sufficiently large $M$. Then we have
\begin{align*}
    \P_N^\theta(\lambda_N \in [x-\delta,x+\delta]) \leq \P_N^\theta(\mc{A}_{x,2\delta}^M) + & \P_N^\theta(d(\hat{\mu}_{X_N},\rho_{\text{sc}} \boxplus \mu_D) > N^{-\kappa}) + \P_N^\theta(\|X_N\| > M).
\end{align*}
An application of Proposition \ref{prop:negligible_all_theta} gives us
\begin{align*}
    \limsup_{N \to \infty} \frac{1}{N} \log \P_N^\theta(\lambda_N \in [x-\delta,x+\delta]) \leq \max\left( \limsup_{N \to \infty} \frac{1}{N} \log \P_N^\theta(\mc{A}_{x,2\delta}^M) , \limsup_{N \to \infty} \frac{1}{N} \log \P_N^\theta(\|X_N\| > M) \right).
\end{align*}
By taking $\delta \downarrow 0$ and applying Lemma \ref{lem:tilted_ub}, we obtain
\begin{align*}
    \limsup_{\delta \downarrow 0} \limsup_{N \to \infty} \frac{1}{N} \log \P_N^\theta(\lambda_N \in [x-\delta,x+\delta]) \leq \max\left(-(I^{(\beta)}(x)-I^{(\beta)}(x,\theta)),\limsup_{N \to \infty} \frac{1}{N} \log \P_N^\theta(\|X_N\| > M) \right).
\end{align*}
Finally we obtain the result by taking $M \to \infty$ and applying again Proposition \ref{prop:negligible_all_theta}.
\end{proof}

\subsection{The proof of the weak LDP lower bound.}\
The following lemma relies on results about the rate function which will be established in Section \ref{sec:rate_fn}. 

\begin{lem}
\label{lem:deviations_likely}
Under the \ref{hyp:Gaussian} Hypothesis, choose any $x \geq \mathtt{r}(\rho_{\text{sc}} \boxplus \mu_D)$; or, under the \ref{hyp:SSGC} Hypothesis, choose any $\mathtt{r}(\rho_{\text{sc}} \boxplus \mu_D) \leq x < x_c$. Then there exists $\theta_x > 0$ such that, for any $M$ sufficiently large depending on $x$ and any $\delta > 0$ sufficiently small depending on $x$, we have
\[
    \lim_{N \to \infty} \frac{1}{N} \log \P_N^{\theta_x}(\mc{A}_{x,\delta}^M) = 0.
\]
If $x < x_c$, then $\theta_x < \theta_c$.
\end{lem}

\begin{proof}
Fix $x \geq \mathtt{r}(\rho_{\text{sc}} \boxplus \mu_D)$, and let $\theta_x$ be such that $I^{(\beta)}(x) = \sup_{\theta \geq 0} I^{(\beta)}(x,\theta) = I^{(\beta)}(x,\theta_x)$. Proposition \ref{prop:study_of_rate_fn} below shows that this exists and is unique (except at $x = \mathtt{r}(\rho_{\text{sc}} \boxplus \mu_D)$, where we choose one of many possible $\theta_x$ values by convention), and that $\theta_x < \theta_c$ whenever $x < x_c$.  We claim that in fact $\P_N^{\theta_x}(\mc{A}_{x,\delta}^M) = 1-o(1)$; to prove this, by Proposition \ref{prop:negligible_all_theta} it suffices to show
\[
    \limsup_{N \to \infty} \frac{1}{N} \log \P_N^{\theta_x} \left(\lambda_N \not\in \left[x-\delta,x+\delta\right]\right) < 0
\]
for $\delta$ small enough. Since $\{\lambda_N < \mathtt{r}(\rho_{\text{sc}} \boxplus \mu_D) - 1\} \subset \{d(\hat{\mu}_{X_N},\rho_{\text{sc}} \boxplus \mu_D) > \epsilon\}$ for some $\epsilon$, and since the law of $\lambda_N$ is exponentially tight under $\P_N^{\theta_x}$, we need only show that for $K$ large enough
\[
    \limsup_{N \to \infty} \frac{1}{N} \log \P_N^{\theta_x} \left( \lambda_N \in \left[\mathtt{r}(\rho_{\text{sc}} \boxplus \mu_D)-1,x-\delta\right] \cup [x+\delta,K]\right) < 0.
\]
But Theorem \ref{thm:wkldpub} shows a weak large deviation upper bound for $\P_N^{\theta_x}$ with the rate function $J^{(\beta)}_x(y) = I^{(\beta)}(y) - I^{(\beta)}(y,\theta_x)$, which Proposition \ref{prop:study_of_rate_fn} below shows is nonnegative and vanishes uniquely at $y = x$. (This theorem applies, since $\theta_x$ is less than $\theta_c$ if necessary.) Since $[\mathtt{r}(\rho_{\text{sc}} \boxplus \mu_D)-1,x-\delta] \cup [x+\delta,K]$ is a compact set that does not contain $x$, this suffices.
\end{proof}

\begin{proof}[Proof of Theorem \ref{thm:wkldplb}]
If $x < \mathtt{r}(\rho_{\text{sc}} \boxplus \mu_D)$, then $I^{(\beta)}(x) = +\infty$, and there is nothing to prove. Thus we will assume in the following that $x \geq \mathtt{r}(\rho_{\text{sc}} \boxplus \mu_D)$.

Whenever $X \in \mc{A}_{x,\delta}^M$, by Corollary \ref{cor:matrix_cty_of_spherical} we have
\[
    J_N(X,\theta_x) \leq 2g_{\theta_x}(\delta) + J(\rho_{\text{sc}} \boxplus \mu_D,\theta_x,x)
\]
for $N$ sufficiently large. In addition, for every $\epsilon > 0$, Lemma \ref{lem:deviations_likely} tells us that for $N$ sufficiently large depending on $\epsilon$ we have $\P_N^{\theta_x}(\mc{A}_{x,\delta}^M) \geq e^{-N\epsilon}$. 

We wish to use Proposition \ref{prop:free_energy} to conclude that, for $N$ sufficiently large depending on $\epsilon$ and on $\theta_x$, we also have
\[
    \E_{X_N}[I_N(X_N,\theta_x)] \geq e^{N(\theta^2+J(\mu_D,\theta_x,\mathtt{r}(\mu_D)) - \epsilon)}.
\]
Under the \ref{hyp:Gaussian} Hypothesis, this is permissible for every $x$; under the \ref{hyp:SSGC} Hypothesis, our restriction $x < x_c$ tells us by Lemma \ref{lem:deviations_likely} that $\theta_x < \theta_c$, so that Proposition \ref{prop:free_energy} indeed applies.

Thus
\begin{align*}
    \P_N(\mc{A}_{x,\delta}^M) &\geq \frac{\E_{X_N}[\mathbf{1}_{X_N \in \mc{A}_{x,\delta}^M} I_N(X_N,\theta_x)]}{\E_{X_N}[I_N(X_N,\theta_x)]} \E_{X_N}[I_N(X_N,\theta_x)] e^{-N \sup_{X \in \mc{A}_{x,\delta}^M} J_N(X,\theta_x)} \\
    &\geq \P_N^{\theta_x}(\mc{A}_{x,\delta}^M)  e^{N(\theta_x^2+J(\mu_D,\theta_x,\mathtt{r}(\mu_D))-\epsilon)} e^{-N \sup_{X \in \mc{A}_{x,\delta}^M} J_N(X,\theta_x)} \\
    &\geq e^{-N\epsilon} e^{N(\theta_x^2+J(\mu_D,\theta_x,\mathtt{r}(\mu_D))-\epsilon)}  e^{-N(J(\rho_{\text{sc}} \boxplus \mu_D, \theta_x, x) + 2g_{\theta_x}(\delta))} \\
    &= e^{-N(I^{(\beta)}(x)+2\epsilon+2g_{\theta_x}(\delta))}.
\end{align*}
Thus, fixing some $M$ sufficiently large, we obtain
\begin{align*}
    \liminf_{N \to \infty} \frac{1}{N} \log\P_N(\abs{\lambda_N(X_N)-x} < \delta) &\geq \liminf_{N \to \infty} \frac{1}{N} \log\P_N(\mc{A}_{x,\delta}^M) \geq -(I^{(\beta)}(x)+2\epsilon+2g_{\theta_x}(\delta))
\end{align*}
and since this is true for every $\epsilon > 0$ we can take the limit as $\delta \downarrow 0$ to conclude.
\end{proof}


\section{Free energy expansion}
\label{sec:free_energy} 

In this section we prove Proposition \ref{prop:free_energy}.

\begin{proof}[Proof under the \ref{hyp:Gaussian} Hypothesis]
Then one can compute directly that
\[
    \E_{X_N}[I_N(X_N,\theta)] = e^{N\theta^2} I_N(D_N,\theta),
\]
so Proposition \ref{prop:free_energy} follows from Proposition \ref{prop:spherical_limit}. (This computation is detailed in the proof of Lemma \ref{lem:free_energy_ub}, below; all inequalities there are actually equalities for the Gaussian ensembles.)
\end{proof}

The proof under the \ref{hyp:SSGC} Hypothesis is more involved and will take up the remainder of this section. We separate the upper and lower bounds as follows:

\begin{lem}
\label{lem:free_energy_ub}
Under the \ref{hyp:SSGC} Hypothesis, for any $\theta \geq 0$ we have
\[
    \limsup_{N \to \infty} \frac{1}{N} \log \E_{X_N}[I_N(X_N,\theta)] \leq \frac{\theta^2}{\beta} + J(\mu_D,\theta,\mathtt{r}(\mu_D)).
\]
\end{lem}

\begin{lem}
\label{lem:free_energy_lb}
Under the \ref{hyp:SSGC} Hypothesis, for any $0 \leq \theta < \theta_c$ we have
\[
    \liminf_{N \to \infty} \frac{1}{N} \log \E_{X_N}[I_N(X_N,\theta)] \geq \frac{\theta^2}{\beta} + J(\mu_D,\theta,\mathtt{r}(\mu_D)).
\]
\end{lem}

The proof of the lower bound will use the following two technical results.

\begin{lem}
\label{lem:laplace_lb_near_zero}
Under the \ref{hyp:SSGC} Hypothesis, for every $\delta > 0$ there exists $\epsilon(\delta) > 0$ such that, for every $N \in \N$, every $i, j \in \llbracket 1, N \rrbracket$, and every $t \in \R$ with $\abs{t} \leq \epsilon(\delta)$ if $\beta = 1$ (or every $t \in \C$ with $\abs{t} \leq \epsilon(\delta)$ if $\beta = 2$), 
\[
    T_{\mu_{i,j}^N}(t) \geq \exp\left((1-\delta)\frac{\abs{t}^2(1+\delta_{ij})}{2\beta}\right).
\]
\end{lem}

\begin{lem}
\label{lem:delocalized_sph_int}
For any $0 \leq \theta < \theta_c$ we have
\begin{equation}
\label{eqn:delocalized_sph_int}
    \lim_{N \to \infty} \frac{1}{N} \log \left( \frac{\E_e\left[\mathbf{1}_{\|e\|_\infty \leq N^{-\frac{3}{8}}} e^{N\theta\ip{e,D_Ne}} \right]}{I_N(D_N,\theta)} \right) = 0.
\end{equation}
\end{lem}

\begin{proof}[Proof of Lemma \ref{lem:free_energy_ub}]
For the remainder of this paper, we introduce the notation
\[
    D_N = \diag(d_1,\ldots,d_N) = \diag(d_1^{(N)},\ldots,d_N^{(N)}).
\]
For every unit vector $e$, by the sub-Gaussian-Laplace-transform assumption of the \ref{hyp:SSGC} Hypothesis we have
\begin{align*}
    \E_{X_N}[e^{N\theta\ip{e,X_Ne}}] &= \left[ \prod_{i < j} T_{\mu_{i,j}^N}(2\sqrt{N}\theta e_i e_j) \right] \left[ \prod_{i=1}^N T_{\mu_{i,i}^N}(\sqrt{N}\theta e_i^2) e^{N\theta d_i e_i^2} \right] \\
    &\leq \left[ \prod_{i < j} \exp\left(2N\theta^2e_i^2e_j^2\right) \right] \left[ \prod_{i=1}^N \exp\left(N\theta^2e_i^4 + N\theta d_ie_i^2\right) \right] \\
    &= \exp\left(N\theta^2\right) \exp\left( N\theta \ip{e,D_Ne} \right).
\end{align*}
To complete the proof, we integrate over $\mathbb{S}^N$ and apply Proposition \ref{prop:spherical_limit}. 
\end{proof}

\begin{proof}[Proof of Lemma \ref{lem:free_energy_lb}]
Fix $\delta > 0$, and let $\epsilon = \epsilon(\delta)$ be as in Lemma \ref{lem:laplace_lb_near_zero}, proved below. Whenever the unit vector $e$ is such that $\|e\|_\infty \leq N^{-3/8}$, we have
\[
    \max_{i,j} \abs{2\sqrt{N}\theta e_ie_j} \leq \epsilon(\delta), \qquad \max_i \abs{\sqrt{N}\theta e_i^2} \leq \epsilon(\delta)
\]
for $N \geq N_0(\delta)$. (The proof below will work with any exponent strictly between $-1/2$ and $-1/4$; but since the exponent does not appear in the final result, we have chosen $-3/8$ for definiteness.) Thus the lower bound on the Laplace transform of the \ref{hyp:SSGC} Hypothesis gives us, for such vectors $e$,
\begin{align*}
    \E_{X_N}[e^{N\theta\ip{e,X_Ne}}] &= \left[ \prod_{i < j} T_{\mu_{i,j}^N}(2\sqrt{N}\theta e_ie_j) \right] \left[ \prod_{i=1}^N T_{\mu_{i,i}^N}(\sqrt{N}\theta e_i^2) e^{N\theta d_ie_i^2} \right] \\
    &\geq \left[ \prod_{i < j} e^{(1-\delta)2N\theta^2 e_i^2e_j^2} \right] \left[ \prod_{i=1}^N e^{(1-\delta) N\theta^2 e_i^4 + N\theta d_ie_i^2} \right] \\
    &= e^{(1-\delta)N\theta^2} e^{N\theta \ip{e,D_Ne}}.
\end{align*}
Therefore
\begin{align*}
    \E_{X_N}[I_N(X_N,\theta)] &= \E_e[\E_{X_N}[e^{N\theta \ip{e,X_Ne}}]] \geq \E_e\left[\mathbf{1}_{\|e\|_\infty \leq N^{-3/8}} \E_{X_N}[I_N(X_N,\theta)]\right] \\
    &\geq e^{(1-\delta)N\theta^2} \E_e\left[\mathbf{1}_{\|e\|_\infty \leq N^{-\frac{3}{8}}} e^{N\theta\ip{e,D_Ne}} \right] = e^{(1-\delta)N\theta^2} \frac{\E_e\left[\mathbf{1}_{\|e\|_\infty \leq N^{-\frac{3}{8}}} e^{N\theta\ip{e,D_Ne}} \right]}{I_N(D_N,\theta)} I_N(D_N,\theta).
\end{align*}
Thus Lemma \ref{lem:delocalized_sph_int}, which is proved below, and Proposition \ref{prop:spherical_limit} give us
\[
    \liminf_{N \to \infty} \frac{1}{N} \log \E_{X_N}[I_N(X_N,\theta)] \geq (1-\delta)\theta^2 + J(\mu_D,\theta,\mathtt{r}(\mu_D))
\]
for every $\delta > 0$.
\end{proof}

\begin{proof}[Proof of Lemma \ref{lem:laplace_lb_near_zero}]
Let $\mu \neq \delta_0$ be a centered measure on $\R$, and write $\mu(f)$ for the integral of a function $f$ against $\mu$. Whenever $x \in \R$, we have $e^x \geq 1+x+\frac{x^2}{2}+\frac{x^3}{6}$; thus
\[
    T_\mu(t) \geq 1+\frac{t^2\mu(x^2)}{2} + \frac{t^3\mu(x^3)}{6} \geq 1+\frac{t^2\mu(x^2)}{2} - \frac{\abs{t}^3\mu(\abs{x}^3)}{6}.
\]
Now it is standard that the bound $T_\mu(t) \leq \exp(\frac{t^2\mu(x^2)}{2})$ implies $\mu(\abs{x}^3) \leq 3(2\mu(x^2))^{3/2}\Gamma(3/2) \leq 8\mu(x^2)^{3/2}$. Then the result follows from the limit
\[
    \lim_{t \to 0} \frac{1}{\abs{t}}\left[\frac{\log\left[1+\frac{t^2\mu(x^2)}{2}-\frac{8\abs{t}^3\mu(x^2)^{3/2}}{6}\right]}{\left(\frac{t^2\mu(x^2)}{2}\right)} - 1\right] = -\frac{8\sqrt{\mu(x^2)}}{3}.
\]
The speed of convergence in this limit can only depend on $\mu$ through $\mu(x^2)$; thus in the result we may choose $\epsilon(\delta)$ uniformly in the distributions $\mu_{i,j}^N$.
\end{proof}

\begin{proof}[Proof of Lemma \ref{lem:delocalized_sph_int}]
This builds on the proof of Lemma 14 in \cite{GuiMai2005}. Notice that the upper bound in Equation \eqref{eqn:delocalized_sph_int} is for free; we only need show the lower bound. 

It is well known that
\[
    (e_1,\ldots,e_N) \overset{d}{=} \left(\frac{g_1}{\|g\|_2},\ldots,\frac{g_N}{\|g\|_2}\right)
\]
where $g = (g_1,\ldots,g_N)$ is a standard Gaussian vector in $\R^N$. The idea is to work in this Gaussian representation, relying on the fact that $\|g\|$ will concentrate around $\sqrt{N}$.

Towards this end, we rewrite our desired inequality as
\[
    \liminf_{N \to \infty} \frac{1}{N} \log \frac{\E\left[ \mathbf{1}_{\frac{\|g\|_\infty}{\|g\|_2} \leq N^{-3/8}} \exp\left(N\theta\frac{\sum_{i=1}^N d_i g_i^2}{\sum_{i=1}^N g_i^2} \right)\right]}{\E\left[\exp\left(N\theta\frac{\sum_{i=1}^N d_i g_i^2}{\sum_{i=1}^N g_i^2}\right)\right]} \geq 0.
\]
Since standard Gaussian measure is isotropic, we may and will assume for the remainder of this proof that the $d_i$'s are ordered as $d_1 \geq \cdots \geq d_N$. Write $v_N$ for the unique solution in $(d_1-\frac{1}{2\theta},+\infty)$ of the equation
\[
    \frac{1}{2\theta} \frac{1}{N} \sum_{i=1}^N \frac{1}{v_N+\frac{1}{2\theta}-d_i} = 1.
\]
(This exists and unique because the left-hand side is a strictly decreasing positive function of $v \in (d_1-\frac{1}{2\theta},+\infty)$, tending to infinity as $v \downarrow d_1-\frac{1}{2\theta}$ and tending to zero as $v \to \infty$.) 

Let us pause to collect some facts about $v_N$. If we write
\[
    d_{\max{}} = d_{\max{}}(N_0) = \sup_{N \geq N_0} \left(\max_{i=1}^N \abs{d_i}\right)
\]
for $N_0$ large enough, then we have \cite[Fact 2.4(3)]{Mai2007} that $v_N \leq d_1 \leq d_{\max{}}$, and by definition $v_N \geq d_1 - \frac{1}{2\theta} \geq -d_{\max{}}-\frac{1}{2\theta}$, so
\begin{equation}
\label{eqn:vn_is_order_one}
    \abs{v_N} \leq d_{\max{}} + \frac{1}{2\theta}.
\end{equation}
Furthermore, the proof of \cite[Theorem 2]{GuiMai2005} shows that, since $\theta < \theta_c$, there exists some small $\eta > 0$ such that
\begin{equation}
\label{eqn:ok_variance}
    \text{for all } i, \qquad 1+2\theta v_N - 2\theta d_i \geq \eta.
\end{equation}
By the proof of \cite[Lemma 14]{GuiMai2005} (for the first inequality) and Equation \eqref{eqn:vn_is_order_one} (for the second), for every $0 < \kappa < \frac{1}{2}$ and $N$ large enough depending on $\kappa$, we have
\begin{equation}
\label{eqn:deloc_first_lb}
\begin{aligned}
    \frac{1}{\E\left[\exp\left(N\theta\frac{\sum_{i=1}^N d_i g_i^2}{\sum_{i=1}^N g_i^2}\right)\right]} &\geq \frac{1}{2} \prod_{i=1}^N \left[ \sqrt{1+2\theta v_N - 2\theta d_i}\right] e^{-N\theta v_N - N^{1-\kappa} \theta(\abs{v_N}+d_{\max{}})} \\
    &\geq \frac{1}{2} \prod_{i=1}^N \left[ \sqrt{1+2\theta v_N - 2\theta d_i}\right] e^{-N\theta v_N - N^{1-\kappa} \theta(2d_{\max{}}+\frac{1}{2\theta})}.
\end{aligned}
\end{equation}

For $0 < \kappa < \frac{1}{2}$, we introduce the event $A_N(\kappa) = \left\{\abs{\frac{\|g\|_2^2}{N}-1} \leq N^{-\kappa}\right\}$. Now the same arguments from \cite[Lemma 14]{GuiMai2005}, along with Equation \eqref{eqn:vn_is_order_one}, give
\begin{align}
\label{eqn:deloc_second_lb}
\begin{split}
    &\E\left[\mathbf{1}_{\frac{\|g\|_\infty}{\|g\|_2} \leq N^{-3/8}} \exp\left(N\theta\frac{\sum_{i=1}^N d_i g_i^2}{\sum_{i=1}^N g_i^2} \right)\right] \\
    &\geq \E\left[\mathbf{1}_{A_N(\kappa)}\mathbf{1}_{\frac{\|g\|_\infty}{\|g\|_2} \leq N^{-3/8}} \exp\left(N\theta\frac{\sum_{i=1}^N d_i g_i^2}{\sum_{i=1}^N g_i^2} \right)\right] \\
    &\geq e^{N\theta v_N -N^{1-\kappa}\theta(d_{\max{}}+\abs{v_N})} \E\left[ \mathbf{1}_{A_N(\kappa)}\mathbf{1}_{\frac{\|g\|_\infty}{\|g\|_2} \leq N^{-3/8}} \exp\left(\sum_{i=1}^N \theta(d_i - v_N)g_i^2 \right) \right] \\
    &= e^{N\theta v_N -N^{1-\kappa}\theta(2d_{\max{}}+\frac{1}{2\theta})} \prod_{i=1}^N \left[ \frac{1}{\sqrt{1+2\theta v_N - 2\theta d_i}} \right] P_N^{v_N}\left(A_N(\kappa), \frac{\|g\|_\infty}{\|g\|_2} \leq N^{-3/8} \right),
\end{split}
\end{align}
where $P_N^{v_N} = P_N^{v_N,D_N,\theta}$ is the probability measure on $\R^N$ defined by
\[
    P_N^{v_N}(\diff g_1, \ldots, \diff g_N) = \frac{1}{\sqrt{2\pi}^N} \prod_{i=1}^N \left[ \sqrt{1+2\theta v - 2\theta d_i^N}e^{-\frac{1}{2}(1+2\theta v - 2\theta d_i^N) g_i^2} \diff g_i\right].
\]
By Equations \eqref{eqn:deloc_first_lb} and \eqref{eqn:deloc_second_lb}, we are done if we can show that
\[
    \lim_{N \to \infty} P_N^{v_N}\left(A_N(\kappa), \frac{\|g\|_\infty}{\|g\|_2} \leq N^{-3/8} \right) = 1.
\]

The proof of \cite[Lemma 14]{GuiMai2005} shows that, for our choice of $v_N$ and since we have chosen $\theta < \theta_c$, we have
\[
    P_N^{v_N}(A_N(\kappa)^c) = o(1),
\]
so it remains only to bound
\begin{align*}
    P_N^{v_N}\left(A_N(\kappa), \frac{\|g\|_\infty}{\|g\|_2} > N^{-3/8}\right) &\leq \sum_{i=1}^N P_N^{v_N}\left(A_N(\kappa),\frac{\abs{g_i}^2}{\|g\|_2^2} \geq N^{-3/4}\right) \\
    &\leq \sum_{i=1}^N P_N^{v_N}\left(\abs{g_i} \geq \sqrt{(N-N^{1-\kappa})N^{-3/4}}\right) \leq \sum_{i=1}^N P_N^{v_N} \left( \abs{g_i} \geq \frac{1}{2} N^{1/8} \right)
\end{align*}
for $N$ large enough depending on $\kappa$. 
But now we observe that $\tilde{g_i} = \sqrt{1+2\theta v_N-2\theta d_i^N}g_i$ are i.i.d. standard normal variables under $P_N^{v_N}$, so that by Equation \eqref{eqn:ok_variance} we have
\begin{align*}
    \sum_{i=1}^N P_N^{v_N}\left(\abs{g_i} \geq \frac{1}{2}N^{1/8}\right) &= \sum_{i=1}^N P_N^{v_N}\left(\abs{\tilde{g_i}} \geq \frac{1}{2}N^{1/8}\sqrt{1+2\theta v_N-2\theta d_i^N}\right) \\
    &\leq NP_N^{v_N}\left(\abs{\tilde{g_i}} \geq \frac{\sqrt{\eta}}{2} N^{1/8} \right) \leq N\exp\left(-\frac{\eta}{8} N^{1/4} \right)
\end{align*}
which is $o(1)$. This concludes the proof.
\end{proof}


\section{Concentration and exponential tightness for tilted measures}
\label{sec:concentration}

\subsection{Proof overview.}\
The proof of Proposition \ref{prop:negligible_all_theta} is broken into the following four lemmata:

\begin{lem}
\label{lem:theta0_suffices}
If Proposition \ref{prop:negligible_all_theta} holds for $\theta = 0$, then it holds for all $\theta > 0$. (For the last point, the same $\gamma > 0$ works for all $\theta \geq 0$.)
\end{lem}

\begin{lem}
\label{lem:theta0_is_true_pt_1}
For any $K > 2d_{\max{}}$, 
\begin{align*}
    \P_N(\lambda_N(X_N) > K) &\leq 4\exp\left( N \left( 5 -\frac{K}{8\sqrt{2}} \right) \right), \\
    \P_N(\lambda_1(X_N) < -K) &\leq 4\exp\left( N \left( 5 - \frac{K}{8\sqrt{2}} \right) \right).
\end{align*}
In particular, the first point of Proposition \ref{prop:negligible_all_theta} is true for $\theta = 0$.
\end{lem}

\begin{lem}
\label{lem:theta0_is_true_pt_2}
The second point of Proposition \ref{prop:negligible_all_theta} is true for $\theta = 0$:
\[
    \lim_{M \to \infty} \limsup_{N \to \infty} \frac{1}{N}\log\P_N(\|X_N\| > M) = -\infty.
\]
\end{lem}

\begin{lem}
\label{lem:replace_with_free_conv}
Under Assumption \ref{assns:quantitative}, the third point of Proposition \ref{prop:negligible_all_theta} is true for $\theta = 0$: There exists $\gamma > 0$ such that, for any $0 < \kappa < \gamma$,
\[
    \lim_{N \to \infty} \frac{1}{N} \log \P_N(d(\hat{\mu}_{X_N},\rho_{\text{sc}} \boxplus \mu_D) > N^{-\kappa}) = -\infty.
\]
Note that this result is the only place in the paper where we use Assumption \ref{assns:quantitative}.
\end{lem}

\subsection{Proof of Lemma \ref{lem:theta0_suffices}.}\
Fix $\theta > 0$. Lemma \ref{lem:free_energy_lb} gives sharp lower bounds on $\E_{X_N}[I_N(X_N,\theta)]$ for subcritical $\theta$ values, but here we need a much weaker lower bound for all positive $\theta$ values. Towards this end, notice that whenever $\mu$ is a centered measure on $\R$, Jensen's gives us $\inf_{t \in \R} T_\mu(t) \geq 1$. Thus for every unit vector $e$ we have
\begin{align}
\label{eqn:weak_free_energy_lb}
    \E_{X_N}[e^{N\theta\ip{e,X_Ne}}]  &= \left[ \prod_{i < j} T_{\mu_{i,j}^N}(2\sqrt{N}\theta e_ie_j) \right] \left[ \prod_{i=1}^N T_{\mu_{i,i}^N}(\sqrt{N}\theta e_i^2) e^{N\theta d_ie_i^2} \right] \geq \prod_{i=1}^N e^{N\theta d_i e_i^2} \geq e^{-N\theta d_{\max{}}}.
\end{align}

Now, whenever $A = A_N$ is a Borel subset of the space of $N \times N$ real matrices, Equation \eqref{eqn:weak_free_energy_lb} and Cauchy-Schwarz give us, for $N$ sufficiently large depending on $\theta$,
\begin{align*}
    \P_N^\theta(A) &= \frac{\E_{X_N}[\mathbf{1}_{X_N \in A}I_N(X_N,\theta)]}{\E_{X_N}[I_N(X_N,\theta)]} \leq e^{N\theta d_{\max{}}} \E_{X_N,e}[\mathbf{1}_{X_N \in A}e^{N\theta\ip{e,X_Ne}}] \\
    &\leq e^{N\theta d_{\max{}}} \sqrt{\P_N(A) \E_{X_N,e}[e^{2N\theta\ip{e,X_Ne}}]} = e^{N\theta d_{\max{}}} \sqrt{\P_N(A) \E_{X_N}[I_N(X_N,2\theta)]} \\
    &\leq e^{N\theta d_{\max{}}} \sqrt{\P_N(A)} e^{2N\left((2\theta)^2 + J(\mu_D,2\theta,\mathtt{r}(\mu_D))\right)}
\end{align*}
where the last inequality follows from Lemma \ref{lem:free_energy_ub}. Thus for any sequence $\{A_N\}$ we have
\begin{align*}
    \limsup_{N \to \infty} \frac{1}{N} \log \P_N^\theta(A_N) \leq \theta d_{\max{}} + \frac{1}{2} \limsup_{N \to \infty} \frac{1}{N} \log \P_N(A_N) + 2\left((2\theta)^2 + J(\mu_D,2\theta,\mathtt{r}(\mu_D))\right).
\end{align*}
This estimate gives us the following two points, from which we can verify the various claims of Proposition \ref{prop:negligible_all_theta} by taking various choices of $\{A_N\}$ and $\{A_{M,N}\}$. 
\begin{itemize}
\item If $\{A_N\}$ is such that $\lim_{N \to \infty} \frac{1}{N} \log \P_N(A_N) = -\infty$, then for all $\theta > 0$ we have 
\[
    \limsup_{N \to \infty} \frac{1}{N}\log \P_N^\theta(A_N) = -\infty.
\]
\item If $\{A_{M,N}\}$ is such that $\lim_{M \to \infty} \limsup_{N \to \infty} \frac{1}{N} \log \P_N^\theta(A_{M,N}) = -\infty$, then for all $\theta > 0$ we have 
\[
    \lim_{M \to \infty} \limsup_{N \to \infty} \frac{1}{N} \log \P_N(A_{M,N}) = -\infty.
\]
\end{itemize}

\subsection{Proof of Lemmas \ref{lem:theta0_is_true_pt_1} and \ref{lem:theta0_is_true_pt_2}.}\
For Lemma \ref{lem:theta0_is_true_pt_1}, notice that it suffices to bound $\P_N(\|X_N\| \geq K)$. But 
\[
    \P_N(\|X_N\| > K) \leq \P_N\left(\left\|\frac{W_N}{\sqrt{N}}\right\| > \frac{K}{2}\right) + \P_N\left(\|D_N\| > \frac{K}{2}\right)
\]
and the second term vanishes for $K$ large enough, so we only need to control the first term. But this was done in \cite[Lemma 1.8]{GuiHui2018}. The constants are slightly worse for the $\beta = 2$ estimate, and we phrase Lemma \ref{lem:theta0_is_true_pt_1} in terms of these worse constants. 

Lemma \ref{lem:theta0_is_true_pt_2} is an immediate consequence.

\subsection{Proof of Lemma \ref{lem:replace_with_free_conv}.}\

\begin{lem}
\label{lem:quant_stieltes_diff}
With $C$ and $\epsilon_0$ as in Assumption \ref{assns:quantitative}, then for any $\eta \leq 1$ we have
\[
    \| G_{\rho_{\text{sc}} \boxplus \hat{\mu}_{D_N}}(E+i\eta) - G_{\rho_{\text{sc}} \boxplus \mu_D}(E+i\eta) \|_{L^\infty(E)} \leq \frac{8\sqrt{C}N^{-\frac{\epsilon_0}{2}}}{\eta^2}.
\]
\end{lem}
\begin{proof}
By recalling the definition of the Dudley distance and by calculating the $L^\infty$ norm and Lipschitz constants of the function $y \mapsto \frac{1}{E+i\eta-y}$, we find that
\[
    \abs{ G_{\rho_{\text{sc}} \boxplus \hat{\mu}_{D_N}}(E+i\eta) - G_{\rho_{\text{sc}} \boxplus \mu_D}(E+i\eta) } \leq \frac{2}{\eta^2} d(\rho_{\text{sc}} \boxplus \hat{\mu}_{D_N}, \rho_{\text{sc}} \boxplus \mu_D),
\]
uniformly in $E \in \R$. 

Now we control $d(\rho_{\text{sc}} \boxplus \hat{\mu}_{D_N},\rho_{\text{sc}} \boxplus \mu_D)$ in terms of $d(\hat{\mu}_{D_N},\mu_D)$. Write $d_L$ for the L\'{e}vy distance between probability measures
\[
    d_L(\mu,\nu) = \inf\{ \epsilon > 0 : \mu(A) \leq \nu(A^\epsilon) + \epsilon \text{ for all Borel A}\}.
\]
Then it is classical \cite[Corollary 11.6.5, Theorem 11.3.3]{Dud2002} that, whenever $\mu$ and $\nu$ are probability measures on $\R$, 
\[
    \frac{1}{2}d(\mu,\nu) \leq d_L(\mu,\nu) \leq 2\sqrt{d(\mu,\nu)}.
\]
On the other hand, \cite[Proposition 4.13]{BerVoi1993} says that
\[
    d_L(\rho_{\text{sc}} \boxplus \hat{\mu}_{D_N}, \rho_{\text{sc}} \boxplus \mu_D) \leq d_L(\rho_{\text{sc}},\rho_{\text{sc}}) + d_L(\hat{\mu}_{D_N},\mu_D) = d_L(\hat{\mu}_{D_N},\mu_D).
\]
Putting these together, we obtain
\begin{align*}
    d(\rho_{\text{sc}} \boxplus \mu_D,\rho_{\text{sc}} \boxplus \hat{\mu}_{D_N}) &\leq 2d_L(\rho_{\text{sc}} \boxplus \mu_D, \rho_{\text{sc}} \boxplus \hat{\mu}_{D_N}) \leq 2d_L(\hat{\mu}_{D_N},\mu_D) \leq 4\sqrt{d(\hat{\mu}_{D_N},\mu_D)}.
\end{align*}
This finishes the proof by Assumption \ref{assns:quantitative}.
\end{proof}

\begin{lem}
\label{lem:using_local_law}
Fix some $A > 0$ independent of $N$. If $\delta > 0$ is chosen sufficiently small, then
\[
    \int_{-A}^A \abs{\E_{X_N}[G_{\hat{\mu}_{X_N}}(E+i N^{-\delta})] - G_{\rho_{\text{sc}} \boxplus \mu_D}(E+i N^{-\delta})} \diff E = O(N^{2\delta-\min(0.99,\frac{\epsilon_0}{2})}).
\]
\end{lem}
\begin{proof}
Throughout, we write $z = E+i\eta$. Later, we will decide how to choose $\eta = \eta(N)$. 

We start by giving an informal overview of the proof. We will compare $\E_{X_N}[G_{\hat{\mu}_{X_N}}(\cdot)]$ and $G_{\rho_{\text{sc}} \boxplus \mu_D}(\cdot)$ via three intermediate comparisons. First, we will import a local law to show that $G_{\hat{\mu}_{X_N}}(z)$ is, with high probability and for appropriate $z$ values, close to the negative normalized trace of a matrix $M_{\text{MDE}}(z) = M_{N,\text{MDE}}(z)$ that exactly solves a matrix equation called the Matrix Dyson Equation (MDE):
\[
    G_{\hat{\mu}_{X_N}}(z) \approx -\frac{1}{N}\tr M_{\text{MDE}}(z).
\]
(The negatives appear since the convention in the local-law literature is to define the Stieltjes transform of a measure as $\int \frac{\mu(\diff y)}{z-y}$ instead of our $\int \frac{\mu(\diff y)}{y-z}$. We have preferred to stick to that convention when working in that vein, so that the reader can more easily cross-reference.) Then we will show that a matrix $M_{\text{Wig}}(z) = M_{N,\text{Wig}}(z)$ whose normalized trace is exactly $-G_{\rho_{\text{sc}} \boxplus \hat{\mu}_{D_N}}$ approximately solves the MDE; standard arguments about the so-called stability of the MDE will then show 
\[
    -\frac{1}{N} \tr M_{\text{MDE}}(z) \approx -\frac{1}{N} \tr M_{\text{Wig}}(z) = G_{\rho_{\text{sc}} \boxplus \hat{\mu}_{D_N}}(z).
\]
Finally, we will use Lemma \ref{lem:quant_stieltes_diff} to show
\[
    G_{\rho_{\text{sc}} \boxplus \hat{\mu}_{D_N}}(z) \approx G_{\rho_{\text{sc}} \boxplus \mu_D}(z).
\]
Notice that all quantities here, except for $G_{\hat{\mu}_{X_N}}$, are deterministic.

For a matrix $M \in \C^{N \times N}$, we define its imaginary part as $\Im(M) = \frac{1}{2i} [M - M^\ast]$. Whenever $\mc{S} : \C^{N \times N} \to \C^{N \times N}$ is a linear operator preserving the set $\{M : \Im(M) > 0\}$, it is known \cite{HelFarSpe2007} that the following constrained equation admits a unique solution:
\begin{equation}
\label{eqn:MDE}
    0 = \Id + (z\Id - D_N + \mc{S}[M(z)])M(z) \qquad \text{subject to} \qquad \Im(M(z)) = \frac{1}{2i}[M(z)-M^\ast(z)] > 0. 
\end{equation}
In particular, we will be interested in the unique solutions to this equation corresponding to two operators $\mc{S}$:
\begin{align*}
    \mc{S}_{\text{MDE}}[M] = \frac{1}{N}\tr(M)\Id + \frac{1}{N}M^T \qquad &\text{induces the solution} \qquad M_{\text{MDE}}(z), \\
    \mc{S}_{\text{Wig}}[M] = \frac{1}{N}\tr(M)\Id \qquad &\text{induces the solution} \qquad M_{\text{Wig}}(z).
\end{align*}
By rearranging Equation \eqref{eqn:MDE} and taking the normalized trace, one can see that $\frac{1}{N}\tr M_{\text{Wig}}(z)$ satisfies the \emph{Pastur equation}
\[
    \frac{1}{N}\tr M_{\text{Wig}}(z) = \int \frac{\hat{\mu}_{D_N}(\diff \lambda)}{\lambda-z-\frac{1}{N}\tr M_{\text{Wig}}(z)},
\]
which characterizes \cite{Pas1972} the Stieltjes transform of $\rho_{\text{sc}} \boxplus \hat{\mu}_{D_N}$. Hence (recall our sign convention)
\[
    -\frac{1}{N}\tr M_{\text{Wig}}(z) = G_{\rho_{\text{sc}} \boxplus \hat{\mu}_{D_N}}(z).
\]
\begin{itemize}
\item For any $\delta > 0$, write $\mathbb{H} = \{z \in \C : \eta > 0\}$ and define the complex domain
\[
    \mc{D}_{\text{far}}^\delta = \{z \in \mathbb{H} : \abs{z} \leq N^{100}, \eta \geq N^{-\delta}\}.
\]
(The notation reminds us that points in this domain are relatively far from the real line; typically in local laws the optimal scale is $\eta \gg \frac{1}{N}$.) Then \cite[Theorem 2.1]{ErdKruSch2019} tells us that there is a universal constant $c > 0$ such that, for any sufficiently small $\epsilon > 0$, there exists $C = C(\epsilon)$ such that
\[
    \P\left(\abs{G_{\hat{\mu}_{X_N}}(z) + \frac{1}{N}\tr(M_{\text{MDE}}(z))} \leq \frac{N^\epsilon}{N} \quad \text{in} \quad \mc{D}_{\text{far}}^{c\epsilon} \right) \geq 1-CN^{-100}.
\]
Since $\frac{1}{N}\tr(M_{\text{MDE}}(z))$ is known by \cite[Proposition 2.1]{AjaErdKru2019} to be the Stieltjes transform of some measure, we also have the trivial bounds $\abs{G_{\hat{\mu}_{X_N}}(E+i\eta)} \leq \frac{1}{\eta}$ and $\abs{\frac{1}{N} \tr(M_{\text{MDE}}(E+i\eta))} \leq \frac{1}{\eta}$. If $\eta = N^a$ for some $-c\epsilon < a < 0$, then for $N$ sufficiently large we have $\{E+ i \eta : \abs{E} \leq A\} \subset \mc{D}_{\text{far}}^{c\epsilon}$; thus whenever $\abs{E} \leq A$ and $\eta$ is as above we have
\begin{align*}
    \E_{X_N} \abs{ G_{\hat{\mu}_{X_N}}(E+i\eta) + \frac{1}{N}\tr M_{\text{MDE}}(E+i\eta)} \leq \frac{N^\epsilon}{N} + \frac{2C}{\eta} N^{-100}.
\end{align*}
so that
\begin{align*}
    \int_{-A}^A \abs{\E_{X_N}[G_{\hat{\mu}_{X_N}}(E+i\eta)] + \frac{1}{N}\tr M_{\text{MDE}}(E+i\eta)} \diff E \leq 2A\left( \frac{N^\epsilon}{N} + \frac{2C}{\eta}N^{-100} \right).
\end{align*}
\item By the definition of $M_{\text{Wig}}$ and since $\mc{S}_{\text{MDE}}[M] = \mc{S}_{\text{Wig}}[M] + \frac{1}{N}M^T$, we have
\begin{align*}
    \Id + (z\Id - D_N + \mc{S}_{\text{MDE}}[M_{\text{Wig}}(z)])M_{\text{Wig}}(z) = \underbrace{\frac{1}{N} M_{\text{Wig}}(z)^TM_{\text{Wig}}(z)}_{=: E(z)}.
\end{align*}
As the notation suggests, we will show that $E(z)$ is an error term, so that $M_{\text{Wig}}(z)$ approximately solves Equation \eqref{eqn:MDE} with $\mc{S} = \mc{S}_{\text{MDE}}$. Indeed, the proof of \cite[Proposition 2.1]{AjaErdKru2019} shows that, for every $z \in \mathbb{H}$, we have
\[
    \|M_{\text{Wig}}(z)\| \leq \frac{1}{\eta}.
\]
In particular, we have
\[
    \|E(z)\| \leq \frac{1}{N\eta^2}.
\]
We will use results of \cite{ErdKruSch2019}, which are phrased in terms of a special matrix norm $\|B\|_{\ast}^{x,y,K}$, depending on $K \in \N$ and $x, y \in \C^N$; the only information we shall need about this norm is that $\abs{\ip{x,By}} \leq \|B\|_{\ast}^{x,y,K}$ for every $x$, $y$, and $K$. If we choose $\eta = N^a$ for some positive or negative $a$, then \cite[Lemma 5.4]{ErdKruSch2019} tells us that, for every $x$, $y$, and $K$, and for every $\epsilon > 0$ and $N \geq N_0(\epsilon,K)$, we have 
\[
    \sup_{x,y} \|E\|_{\ast}^{K,x,y} \leq N^\epsilon \cdot \frac{1}{N\eta^2}.
\]
Thus \cite[Equation 8]{ErdKruSch2019} tells us that, for every $\epsilon > 0$, there exists $\delta(\epsilon) > 0$ and $C_\epsilon > 0$ such that, if $\frac{N^\epsilon}{N\eta^2} \leq N^{-\frac{1}{2K}}$, we have
\begin{align*}
    \sup_{x,y} \|M_{\text{MDE}}(z) - M_{\text{Wig}}(z)\|_\ast^{x,y,K} &\leq C_\epsilon N^{\epsilon + \frac{1}{2K}} \left( \sup_{x,y} \|E\|_{\ast}^{K,x,y} \right) \leq C_\epsilon \frac{N^{2\epsilon + \frac{1}{2K}}}{N\eta^2} \qquad \text{in} \qquad \mc{D}_{\text{far}}^{\delta(\epsilon)}.
\end{align*}
In particular, if $\{e_i\}$ are the standard basis vectors, then for every $\epsilon > 0$, $K \in \N$ and $N \geq N_0(\epsilon,K)$ we have
\begin{align*}
    \abs{\frac{1}{N} \tr M_{\text{MDE}}(z) - \frac{1}{N} \tr M_{\text{Wig}}(z)} &\leq \max_{i=1,\ldots,N} \abs{\ip{e_i,(M_{\text{MDE}}(z) - M_{\text{Wig}}(z))e_i}} \\
    &\leq \max_{i=1,\ldots,N} \|M_{\text{MDE}}(z) - M_{\text{Wig}}(z)\|_{\ast}^{e_i,e_i,K} \\
    &\leq C_\epsilon \frac{N^{2\epsilon + \frac{1}{2K}}}{N\eta^2} && \text{in} \qquad \mc{D}_{\text{far}}^{\delta(\epsilon)}.
\end{align*}
For definiteness, let us choose, say, $\epsilon = \frac{1}{400}$ and $K = 100$, and write $C' = C_{\frac{1}{400}}$ and $\delta = \delta(\frac{1}{400})$; then if $\eta > N^{-\delta}$, for $N$ sufficiently large we have $\{E+i\eta : \abs{E} \leq A\} \subset \mc{D}^\delta_{\text{far}}$ so that
\[
    \int_{-A}^A \abs{\frac{1}{N}\tr M_{\text{MDE}}(E+i\eta) - \frac{1}{N} \tr M_{\text{Wig}}(E+i\eta)}  \diff E \leq 2AC' \cdot \frac{N^{0.01}}{N\eta^2}.
\]
\item If $\eta \leq 1$ then Lemma \ref{lem:quant_stieltes_diff} gives us
\begin{align*}
    \int_{-A}^A \abs{-\frac{1}{N}\tr M_{\text{Wig}}(E+i\eta) - G_{\rho_{\text{sc}} \boxplus \mu_D}(E+i\eta)} \diff E &= \int_{-A}^A \abs{G_{\rho_{\text{sc}} \boxplus \hat{\mu}_{D_N}}(E+i\eta) - G_{\rho_{\text{sc}} \boxplus \mu_D}(E+i\eta)} \diff E \\
    &\leq 16A\sqrt{C} \frac{N^{-\frac{\epsilon_0}{2}}}{\eta^2}.
\end{align*}
\end{itemize}
Combining these estimates, we have the following result: If $\eta = N^{-\delta}$ and $\delta$ is sufficiently small, then every assumption we made on $\eta$ in the above bounds is satisfied and, for all sufficiently small $\epsilon > 0$,
\begin{align*}
    \int_{-A}^A \abs{\E_{X_N}[G_{\hat{\mu}_{X_N}}(E+i\eta)] - G_{\rho_{\text{sc}} \boxplus \mu_D}(E+i\eta)} \diff E &= O\left(\frac{N^{\epsilon}}{N} + \frac{1}{\eta N^{100}} + \frac{N^{0.01}}{N\eta^2} + \frac{1}{N^{\frac{\epsilon_0}{2}}\eta^2} \right) \\
    &= O\left(N^{2\delta-\min(0.99,\frac{\epsilon_0}{2})}\right).
\end{align*}
This concludes the proof.
\end{proof}

\begin{lem}
\label{lem:cdf_diff}
Write 
\begin{align*}
    F_{X_N} &= \hat{\mu}_{X_N}((-\infty,x]), \\
    F_{\rho_{\text{sc}} \boxplus \mu_D} &= (\rho_{\text{sc}} \boxplus \mu_D)((-\infty,x]).
\end{align*}
Then there exists some small $\gamma > 0$ such that
\[
    \sup_x \abs{\E_{X_N}[F_{X_N}(x)]-F_{\rho_{\text{sc}} \boxplus \mu_D}(x)} = O(N^{-\gamma}).
\]
\end{lem}
\begin{proof}
In order to apply a standard technique for bounding Kolmogorov-Smirnov distances, we must first show
\begin{equation}
\label{eqn:cdf_diff_L1}
    \int_{-\infty}^\infty \abs{\E_{X_N}[F_{X_N}(x)]-F_{\rho_{\text{sc}} \boxplus \mu_D}(x)} \diff x < \infty.
\end{equation}
Since $\E_{X_N}[F_{X_N}]$ and $F_{\rho_{\text{sc}} \boxplus \mu_D}$ both take values in $[0,1]$, it suffices to find $M > 0$ such that 
\[
    \int_{\abs{x} > M} \abs{\E_{X_N}[F_{X_N}(x)] - F_{\rho_{\text{sc}} \boxplus \mu_D}(x)} \diff x < \infty.
\]
Furthermore, since $\rho_{\text{sc}} \boxplus \mu_D$ is compactly supported, we may take $M$ so large that $F_{\rho_{\text{sc}} \boxplus \mu_D}(x)$ vanishes for $x < -M$ and is identically one for $x > M$. Now, 
\begin{align}
\label{eqn:cdf_ptwise}
    \E_{X_N}[F_{X_N}(x)] &= \frac{1}{N}\E_{X_N}\left[ \sum_{j=1}^N \mathbf{1}_{\lambda_j(X_N) < x}\right] = \frac{1}{N}\sum_{j=1}^N \P_N(\lambda_j(X_N) < x) \leq \P_N(\lambda_1(X_N) < x)
\end{align}
so that, by Lemma \ref{lem:theta0_is_true_pt_1},
\begin{align*}
    \int_{-\infty}^{-M} \E_{X_N}[F_{X_N}(x)] \diff x &\leq \int_{-\infty}^{-M} 4\exp(N(5+ x/(8\sqrt{2}))) \diff x = \frac{32\sqrt{2}}{N}\exp\left[N\left(5-\frac{M}{8\sqrt{2}}\right)\right]< \infty.
\end{align*}
Similarly, 
\[
    \int_M^\infty (1-\E_{X_N}[F_{X_N}(x)])\diff x < \infty
\]
which finishes the proof of \eqref{eqn:cdf_diff_L1}.

Thus we may import \cite[Theorem 2.2]{Bai1993}, which says that, for any choice of $\eta > 0$ and $B > 0$, we have
\begin{align*}
    \sup_x \abs{\E_{X_N}[F_{X_N}(x)]-F_{\rho_{\text{sc}} \boxplus \mu_D}(x)} \leq \Bigg[ &\frac{1}{\eta} \sup_x \int_{\abs{y} \leq 5\eta} \abs{F_{\rho_{\text{sc}} \boxplus \mu_D}(x+y)-F_{\rho_{\text{sc}} \boxplus \mu_D}(x)} \diff y \\
    &+ \frac{2\pi}{\eta} \int_{\abs{x} > B}  \abs{\E_{X_N}[F_{X_N}(x)]-F_{\rho_{\text{sc}} \boxplus \mu_D}(x)} \diff x \\
    &+ \int_{-10B}^{10B} \abs{\E_{X_N}[G_{\hat{\mu}_{X_N}}(E+i\eta)]-G_{\rho_{\text{sc}} \boxplus \mu_D}(E+i\eta)} \diff E \Bigg].
\end{align*}
We will control the three terms on the right-hand side in order. In the course these estimates we shall choose the parameters $B$ and $\eta = \eta(N)$.
\begin{itemize}
\item Since the compactly supported measure $\rho_{\text{sc}} \boxplus \mu_D$ has $L^\infty$ density \cite[Corollary 5]{Bia1997}, $F_{\rho_{\text{sc}} \boxplus \mu_D}$ is Lipschitz, so we can control the third term by
\[
    \frac{1}{\eta} \sup_x \int_{\abs{y} \leq 5\eta} \abs{F_{\rho_{\text{sc}} \boxplus \mu_D}(x+y)-F_{\rho_{\text{sc}} \boxplus \mu_D}(x)} \diff y \leq 25\eta \|F_{\rho_{\text{sc}} \boxplus \mu_D}(x)\|_{\text{Lip}}.
\]
where $\|f\|_{\text{Lip}} = \sup_{x \neq y} \frac{\abs{f(x)-f(y)}}{\abs{x-y}}$.
\item Choose some $B > \max(\abs{\mathtt{r}(\rho_{\text{sc}} \boxplus \mu_D)},\abs{\mathtt{l}(\rho_{\text{sc}} \boxplus \mu_D)})$; then arguments as above show that
\begin{align*}
    \frac{2\pi}{\eta} \int_{\abs{x} > B} \abs{\E_{X_N}[F_{X_N}(x)]-F_{\rho_{\text{sc}} \boxplus \mu_D}(x)} \diff x \leq \frac{2\pi}{\eta} \cdot \frac{64\sqrt{2}}{N} \exp\left[N\left(5-\frac{B}{8\sqrt{2}}\right)\right].
\end{align*}
Since we will ultimately choose $\eta = N^{-\delta}$ for some small $\delta > 0$, we can choose $B$ so large that this decays exponentially fast.
\item If we choose $\eta = N^{-\delta}$ for $\delta > 0$ sufficiently small, then Lemma \ref{lem:using_local_law} tells us that
\[
    \int_{-10B}^{10B} \abs{\E_{X_N}[G_{\hat{\mu}_{X_N}}(E+i N^{-\delta})] - G_{\rho_{\text{sc}} \boxplus \mu_D}(E+i N^{-\delta})} \diff E = O(N^{2\delta-\frac{\epsilon_0}{2}}).
\]
\end{itemize}
We combine these to obtain
\[
    \sup_x \abs{\E_{X_N}[F_{X_N}(x)] - F_{\rho_{\text{sc}} \boxplus \mu_D}(x)} = O(N^{\max(-\delta,2\delta-\frac{\epsilon_0}{2})}).
\]
\end{proof}

\begin{lem}
\label{lem:concentration_of_empirical_measure}
Under either the \ref{hyp:Gaussian} Hypothesis or the \ref{hyp:SSGC} Hypothesis, there exist positive constants $C_1$ and $C_2$ (depending on the constants in those hypotheses) such that
\[
    \P_N\left[ d(\hat{\mu}_{X_N},\E_{X_N}[\hat{\mu}_{X_N}]) \geq N^{-1/6} \right] \leq C_1 N^{1/4}\exp\left(-C_2N^{7/6}\right).
\]
\end{lem}

\begin{proof}
Concentration results of this type are quite classical, using either the Herbst argument under the log-Sobolev assumption, or results of Talagrand under the compact-support assumption. Indeed, results of the former type are available ``out of the box''; results of the latter type are available ``out of the box'' when $D_N$ vanishes, and we will explain below how to modify the existing proofs for our situation.

Suppose first that we satisfy the log-Sobolev option of the \ref{hyp:SSGC} Hypothesis, that is, that the laws of the entries of $W_N$ satisfy a log-Sobolev inequality with a uniform constant. Since Gaussian measure satisfies the log-Sobolev inequality, the same statement is true under the \ref{hyp:Gaussian} Hypothesis. Furthermore, one can see directly from the definition of the inequality that, if the law of the real random variable $X$ satisfies the logarithmic Sobolev inequality with constant $c$, then for any deterministic $\alpha \in \R$ the law of $X+\alpha$ also satisfies the logarithmic Sobolev inequality with constant $c$. Thus the laws of the entries of $\sqrt{N}X_N$ satisfy a log-Sobolev inequality with uniform constant. This uniformity allows us to import the result \cite[Corollary 1.4b]{GuiZei2000}, which tells us that there exist positive universal constants $C_1$ and $C_2$ such that, for any $\delta > 0$, 
\[
    \P_N[d(\hat{\mu}_{X_N},\E_{X_N}[\hat{\mu}_{X_N}]) \geq \delta] \leq \frac{C_1}{\delta^{3/2}}\exp\left(-C_2N^2\delta^5\right).
\]
By choosing $\delta = N^{-1/6}$, this completes the proof under the \ref{hyp:Gaussian} Hypothesis or under the log-Sobolev option of the \ref{hyp:SSGC} Hypothesis.

Next, we turn to the compact-support option of the \ref{hyp:SSGC} Hypothesis. The barrier to using existing results is that, even if the entries of $W_N$ are uniformly compactly supported, the diagonal entries of
\[
    \sqrt{N}X_N = W_N + \sqrt{N}D_N
\]
are supported in boxes that, while of fixed size, may have centers tending to infinity. So we modify the existing proofs for this situation. Specifically, we start by importing the following result.
\begin{lem}\cite[Theorem 1.3a]{GuiZei2000}
\footnote{
This result was initially stated for centered entries, but by shifting the test function they use to apply \cite[Theorem 6.6]{Tal1996} the proof goes through.
}
\label{lem:gu_ze_concentration_shifted}
Fix $(a_{i,j})_{i, j \leq N} \subset \R^N$, and suppose that there exists a compact set $K \subset \R$ such that the $i,j$th entry of $\sqrt{N}X_N$ is supported on the compact set $a_{i,j} + K = \{a_{i,j} + k : k \in K\}$. Write $\delta_1(N) = 8\abs{K}\sqrt{\pi}/N$. Let $\mc{K} \subset \R$ be compact, and define the class of test functions
\[
    \mc{F}_{\text{lip}, \mc{K}} = \left\{f : \supp(f) \subset \mc{K}, \|f\|_\infty + \sup_{x \neq y} \abs{\frac{f(x)-f(y)}{x-y}} \leq 1\right\}.
\]
Then, for any $\delta \geq 4\sqrt{\abs{\mc{K}}\delta_1(N)}$, we have 
\begin{align*}
    \P\left(\sup_{f \in \mc{F}_{\text{lip}, \mc{K}}} \abs{\tr_N(f(X_N)) - \E[\tr_N(f(X_N))]} \geq \delta \right) \leq \frac{32\abs{\mc{K}}}{\delta} \exp\left(-\frac{N^2}{16\abs{K}^2} \left[ \frac{\delta^2}{16\abs{\mc{K}}} - \delta_1(N)\right]^2 \right).
\end{align*}
\end{lem}
The authors of \cite{GuiZei2000} then extend this result to a supremum over all bounded Lipschitz functions, not just those that are compactly supported, but in the case that $\E[X_N] = 0$. Their arguments require a bound on $\frac{1}{N}\tr(X_N^2)$, which we replace for our model with
\[
    \frac{1}{N}\tr(X_N^2) \leq \sup\{\abs{x}^2 : x \in K\} +d_{\max{}}^2+1,
\]
which is true for $N$ sufficiently large. Following their proofs but substituting this estimate, we obtain the following result, which is analogous to \cite[Corollary 1.4a]{GuiZei2000}:

\begin{lem}
Under the assumptions and notation of the previous lemma, write $S = \sup\{\abs{x}^2 : x \in K\}$ and $M = \sqrt{8(S+d_{\max{}}^2+1)}$. Then for any $N$ sufficiently large and for any $\delta > 0$ satisfying the implicit equation $\delta > (128(M+\sqrt{\delta})\delta_1(N))^{2/5}$, we have
\begin{align*}
    \P_{X_N}\left(d(\hat{\mu}_{X_N},\E_{X_N}(\hat{\mu}_{X_N})) > \delta \right) \leq \frac{128(M+\sqrt{\delta})}{\delta^{3/2}} \exp\left(-\frac{N^2}{16\abs{K}^2}\left[ \frac{\delta^{5/2}}{128(M+\sqrt{\delta})} - \delta_1(N)\right]^2 \right).
\end{align*}
\end{lem}

For $N$ sufficiently large, $\delta = N^{-1/6}$ satisfies the implicit equation given in the lemma, and it is easy to show that
\[
    \left[\frac{\delta^{5/2}}{128(M+\sqrt{\delta})} - \delta_1(N)\right]^2 \geq \frac{N^{7/6}}{N^2(512M)^2}
\]
for $N$ large enough, which gives the desired result in this case.
\end{proof}

\begin{proof}[Proof of Lemma \ref{lem:replace_with_free_conv}]
By Lemma \ref{lem:concentration_of_empirical_measure}, if $\kappa < \frac{1}{6}$ we have
\begin{align*}
    &\P_N(d(\hat{\mu}_{X_N},\rho_{\text{sc}} \boxplus \mu_D) > N^{-\kappa}) \\
    &\leq \mathbf{1}_{d(\E_{X_N}[\hat{\mu}_{X_N}],\rho_{\text{sc}} \boxplus \mu_D) > \frac{N^{-\kappa}}{2}} + \P_N\left(d(\hat{\mu}_{X_N},\E_{X_N}[\hat{\mu}_{X_N}]) > \frac{N^{-\kappa}}{2}\right) \\
    &\leq \mathbf{1}_{d(\E_{X_N}[\hat{\mu}_{X_N}],\rho_{\text{sc}} \boxplus \mu_D) > \frac{N^{-\kappa}}{2}} + C_1N^{1/4}\exp\left(-\frac{C_2}{2}N^{7/6}\right).
\end{align*}
Now we wish to estimate $d(\E_{X_N}[\hat{\mu}_{X_N}],\rho_{\text{sc}} \boxplus \mu_D)$, in order to show that the above indicator vanishes. Towards this end, choose an arbitrary test function $f$ with $\|f\|_{L^\infty}+\sup_{x \neq y} \frac{\abs{f(x)-f(y)}}{\abs{x-y}} \leq 1$. 

First we estimate the tails. For $M$ large enough, Equation \eqref{eqn:cdf_ptwise} gives us
\begin{align*}
    \abs{\int_{-\infty}^{-M} f(x) (\E_{X_N}[\hat{\mu}_{X_N}] - (\rho_{\text{sc}} \boxplus \mu_D))(\diff x)} &= \abs{\int_{-\infty}^{-M} f(x) \E_{X_N}[\hat{\mu}_{X_N}](\diff x)} \leq \|f\|_{L^\infty} \E_{X_N}[F_{X_N}(-M)] \\
    &\leq \E_{X_N}[F_{X_N}(-M)] \leq \P_N(\lambda_1(X_N) < -M) \leq e^{-N}
\end{align*}
where the last inequality follows from Lemma \ref{lem:theta0_is_true_pt_1}. Similarly,
\begin{align*}
    \abs{\int_{M}^{\infty} f(x) (\E_{X_N}[\hat{\mu}_{X_N}] - (\rho_{\text{sc}} \boxplus \mu_D))(\diff x)} \leq 1-\E_{X_N}[F_{X_N}(M)] \leq e^{-N}.
\end{align*}

Thus it remains to estimate $\abs{\int_{-M}^{M} f(x) (\E_{X_N}[\hat{\mu}_{X_N}] - (\rho_{\text{sc}} \boxplus \mu_D))(\diff x)}$. We will do this by approximating $f$ by a test function smooth enough to integrate by parts.

More precisely, suppose first that $f$ is $C^1$ and $\|f'\|_{L^\infty} = \sup_{x \neq y} \frac{\abs{f(x)-f(y)}}{\abs{x-y}} \leq 1$. Then 
\begin{align*}
    \abs{\int_{-M}^{M} f(x) (\E_{X_N}[\hat{\mu}_{X_N}] - (\rho_{\text{sc}} \boxplus \mu_D))(\diff x)} &= \abs{\int_{-M}^M f(x) \diff (\E_{X_N}[F_{X_N}(x)] - F_{\rho_{\text{sc}} \boxplus \mu_D}(x))} \\
    &\leq (2M\|f'\|_{L^\infty} + \|f\|_{L^\infty} + \|f\|_{L^\infty}) \|\E_{X_N}[F_{X_N}] - F_{\rho_{\text{sc}} \boxplus \mu_D}\|_{L^\infty} \\
    &\leq (2M+2)\|\E_{X_N}[F_{X_N}]-F_{\rho_{\text{sc}} \boxplus \mu_D}\|_{L^\infty}.
\end{align*}
Now suppose that $f$ only satisfies $\|f\|_{L^\infty} + \sup_{x \neq y} \frac{\abs{f(x)-f(y)}}{\abs{x-y}} \leq 1$. Since $[-M,M]$ is a compact set independent of $N$, we may choose $g \in C^1$ with $\|g'\|_{L^\infty([-M,M])} \leq 1$ and $\|f-g\|_{L^\infty([-M,M])} \leq (M+1)\|\E_{X_N}[F_{X_N}] - F_{\rho_{\text{sc}} \boxplus \mu_D}\|_{L^\infty}$; thus 
\begin{align*}
    \abs{\int_{-M}^{M} (f(x)-g(x)) (\E_{X_N}[\hat{\mu}_{X_N}] - (\rho_{\text{sc}} \boxplus \mu_D))(\diff x)} &\leq 2\|f-g\|_{L^\infty([-M,M])} \\
    &\leq (2M+2)\|\E_{X_N}[F_{X_N}] - F_{\rho_{\text{sc}} \boxplus \mu_D}\|_{L^\infty}.
\end{align*}
Combining these and and optimizing over $f$, we have
\begin{align*}
    d(\E_{X_N}[\hat{\mu}_{X_N}],\rho_{\text{sc}} \boxplus \mu_D) &\leq 2e^{-N} + 4(M+1)\|\E_{X_N}[F_{X_N}] - F_{\rho_{\text{sc}} \boxplus \mu_D}\|_{L^\infty} = O(N^{-\gamma}),
\end{align*}
where the last equality follows from Lemma \ref{lem:cdf_diff}. Thus if we choose $0 < \kappa < \gamma$, we have 
\[
    \mathbf{1}_{d(\E_{X_N}[\hat{\mu}_{X_N}],\rho_{\text{sc}} \boxplus \mu_D) > \frac{N^{-\kappa}}{2}} = 0
\]
for sufficiently large $N$; in particular this shows us that
\[
    \P_N(d(\hat{\mu}_{X_N},\rho_{\text{sc}} \boxplus \mu_D) > N^{-\kappa}) \leq C_1N^{1/4}\exp\left(-\frac{C_2}{2}N^{7/6}\right)
\]
from which point it is easy to conclude the proof.
\end{proof}


\section{Properties of the rate function}
\label{sec:rate_fn}

The purpose of this section is to show that the supremum in the definition of $I^{(\beta)}(x) = \sup_{\theta \geq 0} I^{(\beta)}(x,\theta)$ is achieved at a value $\theta_x$, which is unique (except for $x = \mathtt{r}(\rho_{\text{sc}} \boxplus \mu_D)$, where it is chosen by convention) and which depends injectively on $x$. This implies that, in the large-deviation upper bound established for tilted measures in Theorem \ref{thm:wkldpub}, the rate function has a unique zero; this property was crucial in the proof of Lemma \ref{lem:deviations_likely} above. At the end of this section, we establish goodness of $I^{(\beta)}(\cdot)$. 

\begin{prop}
\label{prop:study_of_rate_fn}
For every $x > \mathtt{r}(\rho_{\text{sc}} \boxplus \mu_D)$ and for each $\beta = 1, 2$, there exists a unique $\theta \geq 0$, which we will write $\theta_x = \theta_x^{(\beta)}$, such that
\[
    I^{(\beta)}(x) = \sup_{\theta \geq 0} I^{(\beta)}(x,\theta) = I^{(\beta)}(x,\theta_x).
\]
Furthermore, $I^{(\beta)}(x)$ vanishes uniquely at $x = \mathtt{r}(\rho_{\text{sc}} \boxplus \mu_D)$; and if we define by convention 
\[
    \theta_{\mathtt{r}(\rho_{\text{sc}} \boxplus \mu_D)} = \frac{\beta}{2}G_{\rho_{\text{sc}} \boxplus \mu_D}(\mathtt{r}(\rho_{\text{sc}} \boxplus \mu_D))
\]
then the map $x \mapsto \theta_x$ on the domain $\{x \geq \mathtt{r}(\rho_{\text{sc}} \boxplus \mu_D)\}$ is injective. In particular, whenever $x \neq y$ are at least $\mathtt{r}(\rho_{\text{sc}} \boxplus \mu_D)$, we have
\[
    I^{(\beta)}(y) > I^{(\beta)}(y,\theta_x).
\]
We also have
\[
    x_c \geq \mathtt{r}(\rho_{\text{sc}} \boxplus \mu_D)
\]
with equality if and only if $G_{\mu_D}(\mathtt{r}(\mu_D)) = G_{\rho_{\text{sc}} \boxplus \mu_D}(\mathtt{r}(\rho_{\text{sc}} \boxplus \mu_D))$. In addition, 
\[
    \theta_{x_c} = \begin{cases} \frac{\beta}{2}G_{\mu_D}(\mathtt{r}(\mu_D)) & \text{if } G_{\mu_D}(\mathtt{r}(\mu_D)) < +\infty, \\ +\infty & \text{otherwise, by convention} \end{cases}
\]
and if $x < x_c$ then $\theta_x < \theta_{x_c}$. Finally,
\[
    I^{(2)} = 2I^{(1)}.
\]
\end{prop}
\begin{proof}
For the duration of this proof, we introduce the notation
\[
    \mu_D^{\text{sc}} := \rho_{\text{sc}} \boxplus \mu_D.
\]
It can be checked directly from the definition that that, for any compactly supported measure $\nu$ and any $\ms{M} \geq \mathtt{r}(\nu)$, 
\[
    \frac{\partial}{\partial \theta} J(\nu,\theta,\ms{M}) = \begin{cases} R_\nu\left(2\theta\right) & \text{if } 0 \leq 2 \theta \leq G_\nu(\ms{M}), \\ \ms{M}-\frac{1}{2\theta} & \text{if } 2\theta > G_\nu(\ms{M}). \end{cases}
\]
Notice that this is a continuous function of $\theta$. Furthermore, it is known \cite[Lemma 11]{GuiMai2018} that
\[
    G_{\mu_D^{\text{sc}}}(\mathtt{r}(\mu_D^{\text{sc}})) \leq \min(G_{\mu_D}(\mathtt{r}(\mu_D)),G_{\rho_{\text{sc}}}(\mathtt{r}(\rho_{\text{sc}}))) = \min(G_{\mu_D}(\mathtt{r}(\mu_D)),1).
\]
Since $G_{\nu}$ is decreasing on $(\mathtt{r}(\nu),+\infty)$, there are three (or two) phases of $\theta$ values:
\[
    \frac{\partial}{\partial \theta} I^{(1)}(x,\theta) = 
    \begin{cases} 
        R_{\mu_D^{\text{sc}}}\left(2\theta\right) - 2\theta - R_{\mu_D}\left(2\theta\right) = 0 & \text{if } 0 \leq 2\theta \leq G_{\mu_D^{\text{sc}}}(x), \\ 
        x - 2\theta - K_{\mu_D}\left(2\theta\right) & \text{if } G_{\mu_D^{\text{sc}}}(x) \leq 2\theta \leq G_{\mu_D}(\mathtt{r}(\mu_D)), \\ 
        x-2\theta - \mathtt{r}(\mu_D) & \text{if } 2\theta \geq G_{\mu_D}(\mathtt{r}(\mu_D)), 
    \end{cases}
\]
where the third case disappears if $G_{\mu_D}(\mathtt{r}(\mu_D)) = +\infty$ and the second case disappears if $x = \mathtt{r}(\mu_D^{\text{sc}})$ and $G_{\mu_D^{\text{sc}}}(\mathtt{r}(\mu_D^{\text{sc}})) = G_{\mu_D}(\mathtt{r}(\mu_D))$. Notice that this is a continuous function of $\theta \geq 0$, and that, if $G_{\mu_D^{\text{sc}}}(x) \leq \tfrac{2}{\beta}\theta \leq G_{\mu_D^{\text{sc}}}(\mathtt{r}(\mu_D^{\text{sc}}))$, we can in fact write
\[
    \frac{\partial}{\partial \theta} I^{(1)}(x,\theta) = x-K_{\mu_D^{\text{sc}}}\left(2\theta\right).
\]
In general we have
\[
    G_{\mu_D}(\mathtt{r}(\mu_D)) \in [G_{\mu_D^{\text{sc}}}(\mathtt{r}(\mu_D^{\text{sc}})), +\infty].
\]
For the purposes of our analysis, the endpoints of this interval are degenerate cases, and will be handled separately at the end. For now, assume that
\[
    G_{\mu_D}(\mathtt{r}(\mu_D)) \in (G_{\mu_D^{\text{sc}}}(\mathtt{r}(\mu_D^{\text{sc}})), +\infty).
\]
Then $\partial_\theta I^{(1)}(x,\theta)$ has three non-degenerate piecewise sections, and $x_c < \infty$, where we recall the threshold
\[
    x_c = \begin{cases} G_{\mu_D}(\mathtt{r}(\mu_D)) + \mathtt{r}(\mu_D) & \text{if } G_{\mu_D}(\mathtt{r}(\mu_D)) < \infty, \\ +\infty & \text{otherwise}.\end{cases}
\]
In the course of the casework, we will show that $x_c > \mathtt{r}(\mu_D^{\text{sc}})$ in this nondegenerate regime.
\begin{itemize}
\item \textbf{Case 1 ($x < x_c$):} First we study the interval $\theta \in (\frac{1}{2}G_{\mu_D^{\text{sc}}}(x),\frac{1}{2}G_{\mu_D}(\mathtt{r}(\mu_D)))$ and write the function $\partial_\theta I^{(1)}(x,\theta)$ as
\[
    \theta \mapsto f_x(\theta) = x-2\theta - K_{\mu_D}\left(2\theta\right)
\]
defined on this interval. We have
\[
    f''_x(\theta) 
    = 4 \cdot \frac{G''_{\mu_D}\left(K_{\mu_D}\left(2\theta\right)\right)}{\left(G'_{\mu_D}\left(K_{\mu_D}\left(2\theta\right)\right)\right)^3} = -8 \cdot \left( \int \frac{\mu_D(\diff t)}{\left(K_{\mu_D}\left(2\theta\right)-t\right)^2} \right)^{-3} \left( \int \frac{\mu_D(\diff t)}{\left(K_{\mu_D}\left(2\theta\right)-t\right)^3} \right) < 0
\]
since $2\theta < G_{\mu_D}(\mathtt{r}(\mu_D))$, so that $K_{\mu_D}(2\theta) > \mathtt{r}(\mu_D)$ and $\int \frac{\mu_D(\diff t)}{(K_{\mu_D}(2\theta)-t)^i} > 0$ for $i = 2, 3$. Thus $f_x$ is strictly concave.

Let us find out where it is maximized. Since
\[
    f'_x(\theta) = 2\left( \frac{1}{\int \frac{\mu_D(\diff t)}{(K_{\mu_D}(2\theta)-t)^2}} - 1 \right),
\]
we can rearrange
\begin{align*}
    \mathtt{r}(\mu_D^{\text{sc}}) &= K_{\mu_D^{\text{sc}}}(G_{\mu_D^{\text{sc}}}(\mathtt{r}(\mu_D^{\text{sc}}))) \\
    &= R_{\rho_{\text{sc}}}(G_{\mu_D^{\text{sc}}}(\mathtt{r}(\mu_D^{\text{sc}}))) + K_{\mu_D}(G_{\mu_D^{\text{sc}}}(\mathtt{r}(\mu_D^{\text{sc}}))) \\
    &= G_{\mu_D^{\text{sc}}}(\mathtt{r}(\mu_D^{\text{sc}})) + K_{\mu_D}(G_{\mu_D^{\text{sc}}}(\mathtt{r}(\mu_D^{\text{sc}})))
\end{align*}
to obtain
\[
    f'_x\left(\frac{1}{2}G_{\mu_D^{\text{sc}}}(\mathtt{r}(\mu_D^{\text{sc}}))\right) = 2\left( \frac{1}{\int \frac{\mu_D(\diff t)}{(\mathtt{r}(\mu_D^{\text{sc}}) - G_{\mu_D^{\text{sc}}}(\mathtt{r}(\mu_D^{\text{sc}})) - t)^2}} - 1\right).
\]
But it is known that
\[
    \int \frac{\mu_D(\diff t)}{(\mathtt{r}(\mu_D^{\text{sc}}) - G_{\mu_D^{\text{sc}}}(\mathtt{r}(\mu_D^{\text{sc}})) - t)^2} = 1.
\]
Indeed, using the notation and results of \cite[Proposition 2.1]{CapDonFerFev2011} (although the ideas date back to \cite{Bia1997}), the above statement is equivalent to the statement $v_{1,\mu_D}(F_{1,\mu_D}(\mathtt{r}(\mu_D^{\text{sc}}))) = 0$. But $F_{1,\mu_D}$ maps into $\overline{\{u+iv \in \C^+ : v > v_{1,\mu_D}(u)\}}$, and here $F_{1,\mu_D}(\mathtt{r}(\mu_D^{\text{sc}})) = \mathtt{r}(\mu_D^{\text{sc}}) - G_{\mu_D^{\text{sc}}}(\mathtt{r}(\mu_D^{\text{sc}}))$ is real; furthermore $v_{1,\mu_D}(u)$ is a continuous function \cite{Bia1997} of the real parameter $u$. Combined, these conditions force $v_{1,\mu_D}(F_{1,\mu_D}(\mathtt{r}(\mu_D^{\text{sc}}))) = 0$. But this means that
\[
    f'_x \left(\frac{1}{2}G_{\mu_D^{\text{sc}}}(\mathtt{r}(\mu_D^{\text{sc}}))\right) = 0.
\]

Thus we have shown that $f_x(\theta) = \partial_\theta I^{(1)}(x,\theta)$ is a strictly concave function on the open interval \linebreak $(\frac{1}{2}G_{\mu_D^{\text{sc}}}(x),\frac{1}{2}G_{\mu_D}(\mathtt{r}(\mu_D))$, taking a unique maximum value (which can be computed to be $x-\mathtt{r}(\mu_D^{\text{sc}})$) at the point $\theta = \frac{1}{2} G_{\mu_D^{\text{sc}}}(\mathtt{r}(\mu_D^{\text{sc}}))$. Its value at the left endpoint of the interval is $0$, and its value at the right endpoint of the interval is $x-x_c < 0$. In particular, since $f_x$ is decreasing on $(\frac{1}{2}G_{\mu_D^{\text{sc}}}(\mathtt{r}(\mu_D^{\text{sc}})),\frac{1}{2}G_{\mu_D}(\mathtt{r}(\mu_D)))$, taking the value $x-\mathtt{r}(\mu_D^{\text{sc}})$ on the left endpoint and value $x-x_c$ on the right endpoint, we have $x_c > \mathtt{r}(\mu_D^{\text{sc}})$ as claimed.

Now if $\theta \geq \frac{1}{2}G_{\mu_D}(\mathtt{r}(\mu_D))$, then
\[
    \partial_\theta I^{(1)}(x,\theta) = x-2\theta-\mathtt{r}(\mu_D) < x_c - G_{\mu_D}(\mathtt{r}(\mu_D)) - \mathtt{r}(\mu_D) = 0.
\]

There are two subcases here:
\begin{itemize}
\item \textbf{Subcase a ($x = \mathtt{r}(\mu_D^{\text{sc}})$):} 
Here, $f_x(\theta) = \partial_\theta I^{(1)}(x,\theta)$ takes maximum value $x-\mathtt{r}(\mu_D^{\text{sc}}) = 0$ on the open interval $(\frac{1}{2}G_{\mu_D^{\text{sc}}}(\mathtt{r}(\mu_D^{\text{sc}})),\frac{1}{2}G_{\mu_D}(\mathtt{r}(\mu_D)))$, and is negative on the interval $[\frac{1}{2}G_{\mu_D}(\mathtt{r}(\mu_D)),+\infty)$. Thus $I^{(1)}(\mathtt{r}(\mu_D^{\text{sc}})) = 0$.

\item \textbf{Subcase b ($x > \mathtt{r}(\mu_D^{\text{sc}})$):}
Here, the value of the function $f_x$ at $\theta = \frac{1}{2}G_{\mu_D^{\text{sc}}}(\mathtt{r}(\mu_D^{\text{sc}}))$ is $x-\mathtt{r}(\mu_D^{\text{sc}}) > 0$. Thus it vanishes at a unique point $\theta_x \in (\frac{1}{2}G_{\mu_D^{\text{sc}}}(\mathtt{r}(\mu_D^{\text{sc}})),\frac{1}{2}G_{\mu_D}(\mathtt{r}(\mu_D)))$. For such values of $x$, then, $I^{(1)}(x,\theta)$ vanishes for $\theta \in [0,\frac{1}{2}G_{\mu_D^{\text{sc}}}(x)]$; strictly increases for $\theta \in (\frac{1}{2}G_{\mu_D^{\text{sc}}},\theta_x)$; and strictly decreases for $\theta \in (\theta_x,+\infty)$. In particular $I^{(1)}(x) > 0$ for such $x$ values.
\end{itemize}
\item \textbf{Case 2 ($x \geq x_c$):} 
Here we can explicitly write
\begin{equation}
\label{eqn:explicit_thetax}
    \theta_x = \frac{1}{2}(x-\mathtt{r}(\mu_D)).
\end{equation}
The function $f_x$ defined above is still strictly concave on its domain and still vanishes at the left endpoint of this domain, but now its value at the right endpoint is nonnegative; thus $I^{(1)}(x,\theta)$ is strictly increasing for $\theta \in (\frac{1}{2}G_{\mu_D^{\text{sc}}}(x),\frac{1}{2}G_{\mu_D}(\mathtt{r}(\mu_D)))$. A simple analysis of $\partial_\theta I^{(1)}(x,\theta)$ for $\theta \geq \frac{1}{2}G_{\mu_D}(\mathtt{r}(\mu_D))$ shows that $\theta_x$ as defined above is, as claimed, the unique $\theta$ value that maximizes $I^{(1)}(x,\theta)$, and $I^{(1)}(x) > 0$.

In particular notice that
\[
    \theta_{x_c} = \frac{1}{2}(x_c-\mathtt{r}(\mu_D)) = \frac{1}{2}G_{\mu_D}(\mathtt{r}(\mu_D)).
\]
\end{itemize}
It remains only to show that $x_1 \neq x_2 \implies \theta_{x_1} \neq \theta_{x_2}$. If $x_1 < x_c \leq x_2$, then $\theta_{x_1}$ and $\theta_{x_2}$ as constructed above lie in disjoint intervals, so cannot be equal; and if $x_c \leq x_1, x_2$ then we can see $\theta_{x_1} \neq \theta_{x_2}$ from our explicit formula \eqref{eqn:explicit_thetax}. Thus we only need consider $x_1 < x_2 < x_c$. If $x_1 = \mathtt{r}(\mu_D^{\text{sc}})$, then $\theta_{x_1} = \frac{1}{2}G_{\mu_D^{\text{sc}}}(\mathtt{r}(\mu_D^{\text{sc}})) < \theta_{x_2}$ by construction; thus we can assume $\mathtt{r}(\mu_D^{\text{sc}}) < x_1 < x_2 < x_c$. But then $\theta_{x_1}$ and $\theta_{x_2}$ are defined on the common interval $(\frac{1}{2}G_{\mu_D^{\text{sc}}}(\mathtt{r}(\mu_D^{\text{sc}})), \frac{1}{2}G_{\mu_D}(\mathtt{r}(\mu_D)))$ as the unique points satisfying
\[
    2\theta_{x_1} + K_{\mu_D}\left(2\theta_{x_1}\right) = x_1 \neq x_2 = 2\theta_{x_2} + K_{\mu_D}\left(2\theta_{x_2}\right).
\]
Thus we must have $\theta_{x_1} \neq \theta_{x_2}$. 

Now we explain the necessary adjustments in the degenerate cases.
\begin{itemize}
\item \textbf{Degenerate Case 1 ($G_{\mu_D}(\mathtt{r}(\mu_D)) = G_{\mu_D^{\text{sc}}}(\mathtt{r}(\mu_D^{\text{sc}}))$):}

The proof of \cite[Lemma 11]{GuiMai2018} shows that $\omega(\mathtt{r}(\mu_D^{\text{sc}})) \geq \mathtt{r}(\mu_D)$, where $\omega$ is defined (see \cite[Proposition 2.1]{CapDonFerFev2011}) as $\omega(z) = z - G_{\mu_D^{\text{sc}}}(z)$; hence 
\begin{align*}
    x_c &= \mathtt{r}(\mu_D) + G_{\mu_D}(\mathtt{r}(\mu_D)) = \mathtt{r}(\mu_D) + G_{\mu_D^{\text{sc}}}(\mathtt{r}(\mu_D^{\text{sc}})) \\
    &= \mathtt{r}(\mu_D) + \mathtt{r}(\mu_D^{\text{sc}}) - \omega(\mathtt{r}(\mu_D^{\text{sc}})) \leq \mathtt{r}(\mu_D^{\text{sc}})
\end{align*}
and all $x$ are ``at least critical.''
\begin{itemize}
\item \textbf{Degenerate Subcase a ($x = \mathtt{r}(\mu_D^{\text{sc}})$):} Then we only have
\begin{align*}
    \partial_\theta I^{(1)}(\mathtt{r}(\mu_D^{\text{sc}}),\theta) = \begin{cases} 0 & \text{if } 0 \leq 2\theta \leq G_{\mu_D}(\mathtt{r}(\mu_D)) \\ \mathtt{r}(\mu_D^{\text{sc}}) - 2\theta - \mathtt{r}(\mu_D) & \text{if } 2\theta \geq G_{\mu_D}(\mathtt{r}(\mu_D)). \end{cases}
\end{align*}
From the first line of this display and from the equality $G_{\mu_D}(\mathtt{r}(\mu_D)) = G_{\mu_D^{\text{sc}}}(\mathtt{r}(\mu_D^{\text{sc}}))$ we have 
\begin{equation}
\label{eqn:degenerate}
    0 = R_{\mu_D^{\text{sc}}}(G_{\mu_D^{\text{sc}}}(\mathtt{r}(\mu_D^{\text{sc}}))) - G_{\mu_D}(\mathtt{r}(\mu_D)) - R_{\mu_D}(G_{\mu_D}(\mathtt{r}(\mu_D))) = \mathtt{r}(\mu_D^{\text{sc}}) - G_{\mu_D}(\mathtt{r}(\mu_D)) - \mathtt{r}(\mu_D).
\end{equation}
On the one hand, \eqref{eqn:degenerate} tells us that 
\[
    x_c = \mathtt{r}(\mu_D^{\text{sc}})
\]
so that by convention 
\[
    \theta_{x_c} = \frac{1}{2}G_{\mu_D^{\text{sc}}}(\mathtt{r}(\mu_D^{\text{sc}})) = \frac{1}{2}G_{\mu_D}(\mathtt{r}(\mu_D))
\]
as claimed. On the other hand, if $2\theta \geq G_{\mu_D}(\mathtt{r}(\mu_D))$ then \eqref{eqn:degenerate} tells us that
\begin{align*}
    \partial_\theta I^{(1)}(\mathtt{r}(\mu_D^{\text{sc}},\theta)) &= \mathtt{r}(\mu_D^{\text{sc}}) - 2 \theta - \mathtt{r}(\mu_D) \leq \mathtt{r}(\mu_D^{\text{sc}}) - G_{\mu_D}(\mathtt{r}(\mu_D)) - \mathtt{r}(\mu_D) = 0.
\end{align*}
So $\partial_\theta I^{(1)}(\mathtt{r}(\mu_D^{\text{sc}})) \leq 0$ for all $\theta$ and $I^{(1)}(\mathtt{r}(\mu_D^{\text{sc}})) = 0$ as claimed.
\item \textbf{Degenerate Subcase b ($x > \mathtt{r}(\mu_D^{\text{sc}})$):} Then $f_x$ as above is defined and strictly concave on a nondegenerate interval; it vanishes at the left endpoint of this interval; it takes a positive maximum (namely $x-\mathtt{r}(\mu_D^{\text{sc}})$) at the right endpoint of this interval. Thus the analysis of Case 2 above holds to show that $\theta_x$ is given by Equation \eqref{eqn:explicit_thetax}.
\end{itemize}
The argument above for injectivity goes through, since Equation \eqref{eqn:explicit_thetax} works for all $x$ values.
\item \textbf{Degenerate Case 2 ($G_{\mu_D}(\mathtt{r}(\mu_D)) = +\infty$):} Here $x_c = +\infty$, and all $x$ values are subcritical. The function $f_x$ from Case 1 is then defined and strictly concave on the interval $(\frac{1}{2}G_{\mu_D^{\text{sc}}}(x),+\infty)$. It has a unique maximum at $\frac{1}{2}G_{\mu_D^{\text{sc}}}(\mathtt{r}(\mu_D^{\text{sc}}))$, where its value is positive; and strict concavity tells us $\lim_{\theta \to +\infty}f_x(\theta) = -\infty$; thus $f_x$ still has a unique zero on its domain, which we still call $\theta_x$. The argument above for injectivity goes through.
\end{itemize}
Now we show $I^{(2)} = 2I^{(1)}$. If $x < x_c$, then we showed that $\theta_x$ is defined implicitly by 
\[
    \frac{2}{\beta} \theta_x + K_{\mu_D}\left(\frac{2}{\beta} \theta_x\right) = x \qquad \text{subject to} \qquad \frac{2}{\beta}\theta_x \in \left(G_{\mu_D^{\text{sc}}}(\mathtt{r}(\mu_D^{\text{sc}})),G_{\mu_D}(\mathtt{r}(\mu_D))\right).
\]
If $x \geq x_c$, then we have
\[
    \frac{2}{\beta}\theta_x = x-\mathtt{r}(\mu_D).
\]
Notice that $\frac{2}{\beta}\theta_x$ is independent of $\beta$. But then the definition \eqref{eqn:limitJ} gives us
\[
    J^{(\beta=2)}(\nu,\theta_x,\ms{M}) = 2J^{(\beta=1)}(\nu,\theta_x,\ms{M})
\]
from which the claim follows.
\end{proof}

\begin{prop}
The function $I^{(\beta)}(\cdot)$ is a good rate function.
\end{prop}
\begin{proof}
First, for any compactly-supported measure $\mu$ and any $\lambda \geq \mathtt{r}(\mu)$, we have $J(\mu,0,\lambda) = 0$; hence $I^{(1)}(x)$ is nonnegative.

For every fixed $\theta$, dominated convergence tells us that $J(\rho_{\text{sc}} \boxplus \mu_D, \theta, x)$ is a continuous function of $x > \mathtt{r}(\rho_{\text{sc}})$; hence $I^{(1)}(\cdot)$ is lower semi-continuous at such $x$ values. It is also lower semi-continuous for $x < \mathtt{r}(\rho_{\text{sc}} \boxplus \mu_D)$, where its value is infinite. Finally, since $I^{(1)}(\cdot)$ is nonnegative and vanishes at $\mathtt{r}(\rho_{\text{sc}} \boxplus \mu_D)$, it is also lower semi-continuous there. 

Hence $I^{(1)}(\cdot)$ is a rate function. But since $I^{(1)}(\cdot)$ is the rate function for a weak LDP of an exponentially tight family, it is classical (see, e.g., \cite[Lemma 1.2.18]{DemZei1998}) that $I^{(1)}$ is in fact good.
\end{proof}

\bibliographystyle{apalike}
\bibliography{ldpbib.bib}

\end{document}